\documentclass{article}
\usepackage[T1]{fontenc}
\usepackage[utf8]{inputenc}
\usepackage{authblk}

\title{Balanced Excited Random Walk\\ in Two Dimensions}

\author[1]{Omer Angel\thanks{angel@math.ubc.ca}}
\author[2]{Mark Holmes\thanks{holmes.m@unimelb.edu.au}}
\author[3]{Alejandro Ram\'\i rez\thanks{aramirez@mat.uc.cl}}
\affil[1]{Department of Mathematics,   the University of
  British Columbia}
\affil[2]{School of Mathematics and Statistics,  the University 
  of Melbourne}
\affil[3]{Facultad de Matem\'aticas,  Pontificia Universidad
  Cat\'olica
  de Chile}

\date{\today}

\usepackage{amsmath,amssymb,amsfonts,amsthm}

\usepackage[colorlinks=true,linkcolor=black,citecolor=black,urlcolor=blue,pdfborder={0 0 0}]{hyperref}

\usepackage{graphicx}
\usepackage{enumitem}
\usepackage{color}

\usepackage[margin=1cm]{caption}
\usepackage[protrusion=true
]{microtype}

\usepackage[nameinlink]{cleveref}
  \crefname{theorem}{Theorem}{Theorems}
  \crefname{thm}{Theorem}{Theorems}
  \crefname{mainthm}{Theorem}{Theorems}
  \crefname{lemma}{Lemma}{Lemmas}
  \crefname{lem}{Lemma}{Lemmas}
  \crefname{remark}{Remark}{Remarks}
  \crefname{prop}{Proposition}{Propositions}
  \crefname{defn}{Definition}{Definitions}
  \crefname{corollary}{Corollary}{Corollaries}
  \crefname{section}{Section}{Sections}
  \crefname{figure}{Figure}{Figures}

\newtheorem{thm}{Theorem}[section]
\newtheorem{theorem}[thm]{Theorem}

\newtheorem{lemma}[thm]{Lemma}
\newtheorem{corollary}[thm]{Corollary}
\newtheorem{prop}[thm]{Proposition}

\newtheorem{remark}[thm]{Remark}

\newtheorem{open}[thm]{Open Problem}

\numberwithin{equation}{section}

\usepackage{bbm}
\usepackage[textsize=scriptsize]{todonotes}
\usepackage{mathtools}
 \newcommand{\indic}[1]{\mathbbm{1}_{\left\{#1\right\}}}

\DeclareMathOperator{\Bin}{Bin}
\newcommand{\vep}{\varepsilon}
\newcommand{\mc}[1]{\mathcal{#1}}
\newcommand{\cL}{\mathcal{L}}

\newcommand{\Maaak}[1]{{\color{purple} #1}}

\newcommand{\bs}[1]{\boldsymbol{#1}}
\newcommand{\sss}[1]{\scriptscriptstyle{#1}}

\newcommand{\floor}[1]{\lfloor #1 \rfloor}

\newcommand{\blank}[1]{}
\newcommand{\bp}{\bs{P}}
\newcommand{\bq}{\bs{Q}}
\newcommand{\LL}{\mathrm{L}}
\newcommand{\ord}[1]{\overset{{\scriptscriptstyle #1}}{\smile}}
\newcommand{\Geom}{\mathrm{Geom}}
\newcommand{\Neg}{\mathrm{Neg}}
\newcommand{\ouch}{\lceil 2np \rceil}

\renewcommand{\P}{\mathbb P}
\newcommand{\Z}{\mathbb Z}
\newcommand{\E}{\mathbb E}
\newcommand{\R}{\mathbb R}
\newcommand{\N}{\mathbb N}
\newcommand{\eps}{\varepsilon}
\newcommand{\F}{\mathcal{F}}

\newcommand{\cS}{\mathcal S}

\newcommand{\cB}{\mathcal B}
\newcommand{\cA}{\mathcal A}
\newcommand{\ssX}{{\sss{\bs{X}}}}
\newcommand{\ssY}{{\sss{\bs{Y}}}}
\newcommand{\ssZ}{{\sss{\bs{Z}}}}

\begin{document}
\maketitle
\abstract{We give non-trivial upper and lower bounds on the range of the so-called Balanced Excited Random Walk in two dimensions, and verify a conjecture of Benjamini, Kozma and Schapira.  To the best of our knowledge these are the first non-trivial results for this 2-dimensional model.}

\vspace{0.5cm}



\section{Introduction and main results}
\label{sec:intro}

We consider the following model of a random walk in $\Z^2$.
On the first departure from each site $x$, the walk makes a vertical step.
On all subsequent departures from the vertex the walk makes a horizontal step.
In each case, the step is to one of the two neighbours with equal probability, independent of all previous randomness, so that the walk $\bs{Z}=(Z_n)_{n \ge 0}$ is a martingale.  Figure \ref{fig:simulation} shows a simulation of this process, which we also refer to as BERW, up to time $n=10^8$.

\begin{figure}
\includegraphics[scale=0.4]{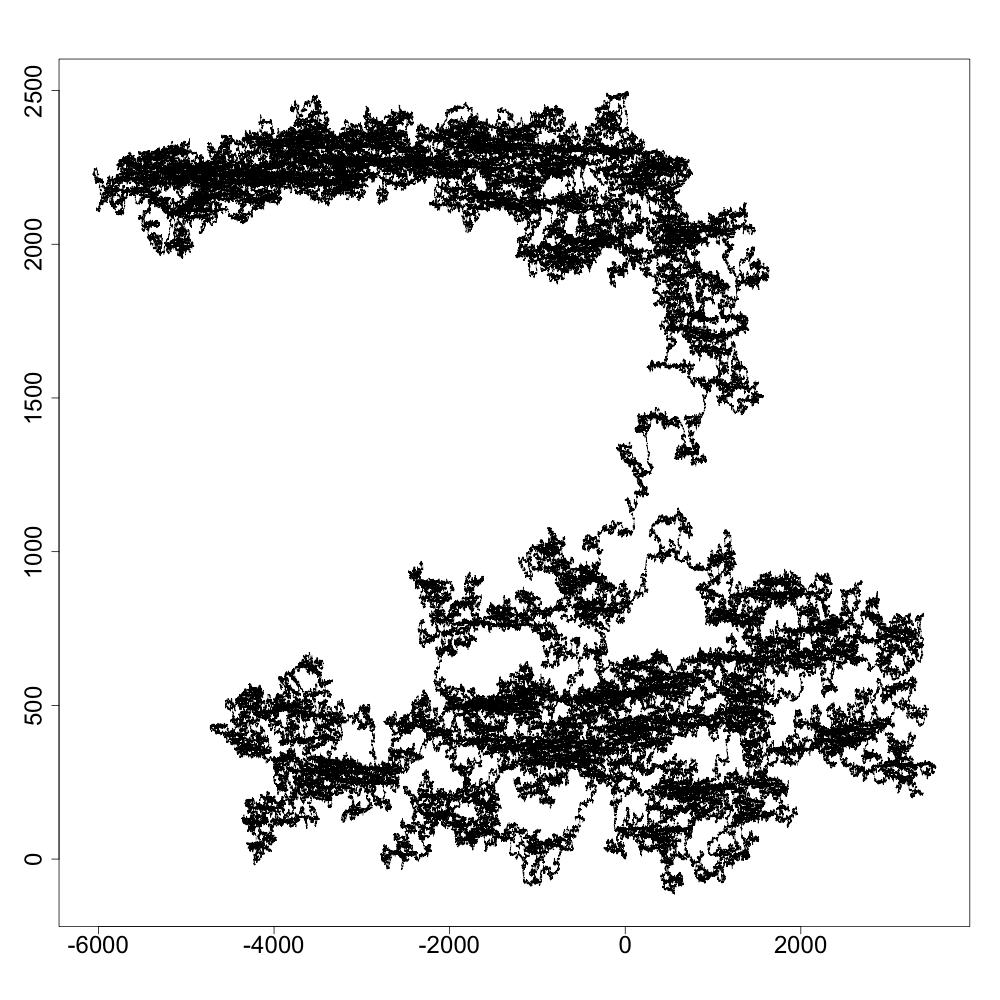}
\caption{A simulation of the path of the 2-dimensional balanced excited random walk, up to time $n=10^8$}
\label{fig:simulation}
\end{figure}

This model was introduced in more general dimensions $d$ by Benjamini, Kozma and Schapira in \cite{BKS11}: on the first departure from any site a simple random walk step in the first $d_1$ coordinate directions is taken, while on subsequent departures a simple random walk step in the last $d_2$ coordinate directions is taken, and $d=d_1+d_2$.  If $d_1\vee d_2\ge 3$ then the walk is trivially transient.  They proved that in $4=2+2$ dimensions the walk is transient, and they stated various conjectures and beliefs about lower dimensional settings.  In 3 dimensions, where $3=2+1$ and $3=1+2$ dimensions they conjectured that the walk would be transient.  This was proved by Peres,  Schapira and Sousi in 2016 \cite{PSS16} for the
case $3=1+2$.  

In the 2 dimensional case $2=1+1$, Benjamini et al.~\cite{BKS11} conjecture that the walk is recurrent, in the sense that every vertex is visited infinitely often (a.s.).  To the best of our knowledge, nothing non-trivial has been proved about this model.  As the walk is non-Markovian, it is not even clear that the set of recurrent sites (i.e.~those visited infinitely often by the walk) is empty or all of $\Z^2$.  Our first main result, Theorem \ref{thm:dichotomy} verifies this statement.

We say that a vertex $x\in\Z^2$ is recurrent for a trajectory $\bs{Z}=(Z_n)_{n\ge 0}$ if it is visited infinitely often by the trajectory (i.e.~$x=Z_i$ for infinitely many $i \in \Z_+:=\N\cup \{0\}$). Let $D$ denote the set of recurrent vertices of $\bs{Z}$.
\begin{theorem}
  \label{thm:dichotomy}
  For the balanced excited random walk in $2=1+1$ dimensions, 
  \begin{equation}
  \P(D\in \{\varnothing, \Z^2\})=1.
  \end{equation}
\end{theorem}
Note that Theorem \ref{thm:dichotomy} does not rule out the possibility that $D$ is random.  In agreement with Benjamini et al.~we conjecture that the walk is recurrent, i.e.~$\P(D=\Z^2)=1$.  A less ambitious goal would be to prove the following 0-1 law.
\begin{open}
  For the balanced excited random walk in $2=1+1$ dimensions, either $\P(D=\varnothing)=1$ or $\P(D=\Z^2)=1$.
\end{open}

Benjamini et al.~also make a conjecture about the limiting shape and growth of the \emph{range}, $\mathcal{R}_n=\{Z_0,\dots,Z_{n}\}$.  One expects that to leading order $R_n:=\#\mathcal{R}_n$ grows like $n^\alpha$ for some $\alpha\in [1/2,1]$.  These two bounds on $\alpha$ are trivial since $R_n\le n+1$, and  the projection of the walk in the direction $(1,1)$ is a 1-dimensional simple random walk (similarly for $(-1,1)$).  

\begin{open}
Prove that $\alpha$ exists, e.g.~prove that there exists $\alpha\in [1/2,1]$ such that (convergence in probability) as $n \to \infty$,
\[\dfrac{\log(R_n)}{\log n} \overset{\P} {\to}\alpha.\]
\end{open}

Rudimentary numerical analysis suggests that $\alpha \approx 0.78$.

\begin{open}
Provide a better numerical estimate of $\alpha$.
\end{open}

The list of horizontal moves of $\bs{Z}$ forms a 1-dimensional simple random walk, as does the list of vertical moves, and these two lists are independent.  The number of horizontal steps by time $n$ must grow linearly with $n$ since e.g.~any two consecutive vertical steps of opposite sign must either have a horizontal step between them or be immediately followed by a horizontal step.  Thus the ``horizontal range'' ($\max_{k\le n} Z_k\cdot (1,0)-\min_{k\le n}Z_k \cdot(1,0)$) grows like $\sqrt{n}$.  Note that the range of the walk at time $n$ is exactly one plus the number of vertical steps taken (a vertical step is taken when and only when departing from a newly visited site).  Therefore, if the growth of $R_n$ is sublinear then the number of vertical steps taken up to time $n$ is sublinear (so the vertical range is $o(\sqrt{n})$) and the limiting shape of the range will be a horizontal line.  Benjamini et al.~\cite[Page 2]{BKS11} conjecture that this is the case. Our second main result provides non-trivial lower and upper bounds on the growth of the range and in particular verifies this conjecture.

\begin{theorem}
  \label{thm:range}
  As $n \to \infty$,
  \begin{itemize}
 \item[\emph{(i)}] for any sequence $a_n \to \infty$, $\P\Big(R_n> \dfrac{n a_n}{\sqrt{\log\log n}}\Big)\to 0$, and 
 \item[\emph{(ii)}]  $\P\big(R_n\ge n^{4/7}/(\log^2n)\big)\to 1$.
 \end{itemize}
\end{theorem}
Roughly speaking, the first claim says that the range does not grow faster than $n/\sqrt{\log\log n}$, while the second says that it grows at least as fast as $n^{4/7}/\log^2 n$.  This result gives rise to the following Corollary (see also Figures \ref{fig:common_scale} and \ref{fig:components}).

\begin{corollary}
  \label{cor:BM}
  For every $T>0$ we have that $(n^{-1/2}Z_{\lfloor nt\rfloor})_{t \in [0,T]}$ converges weakly to $((B_t,0_t))_{t \in [0,T]}$ as $n \to \infty$, where $B_t$ is a standard 1-dimensional Brownian motion and $0_t=0$ for all $t$.
\end{corollary}

Neither of the bounds in Theorem \ref{thm:range} are very close to our crude numerical estimate $\alpha\approx 0.78$.  It would be of interest to improve either of these bounds, e.g., by proving resolving one or both of the following.
\begin{open}
Prove that for some $\overline{\alpha}<1$ one has:  
\[\P(R_n> n^{\overline{\alpha}})\to 0, \text{ as }n \to \infty.\]
\end{open}
\begin{open}
Prove that for some $\underline{\alpha}>4/7$ one has:  
\[\P\big(R_n\ge n^{\underline{\alpha}}
\big)\to 1, \text{ as }n \to \infty.\]
\end{open}

\begin{figure}
\begin{center}
\includegraphics[scale=0.2]{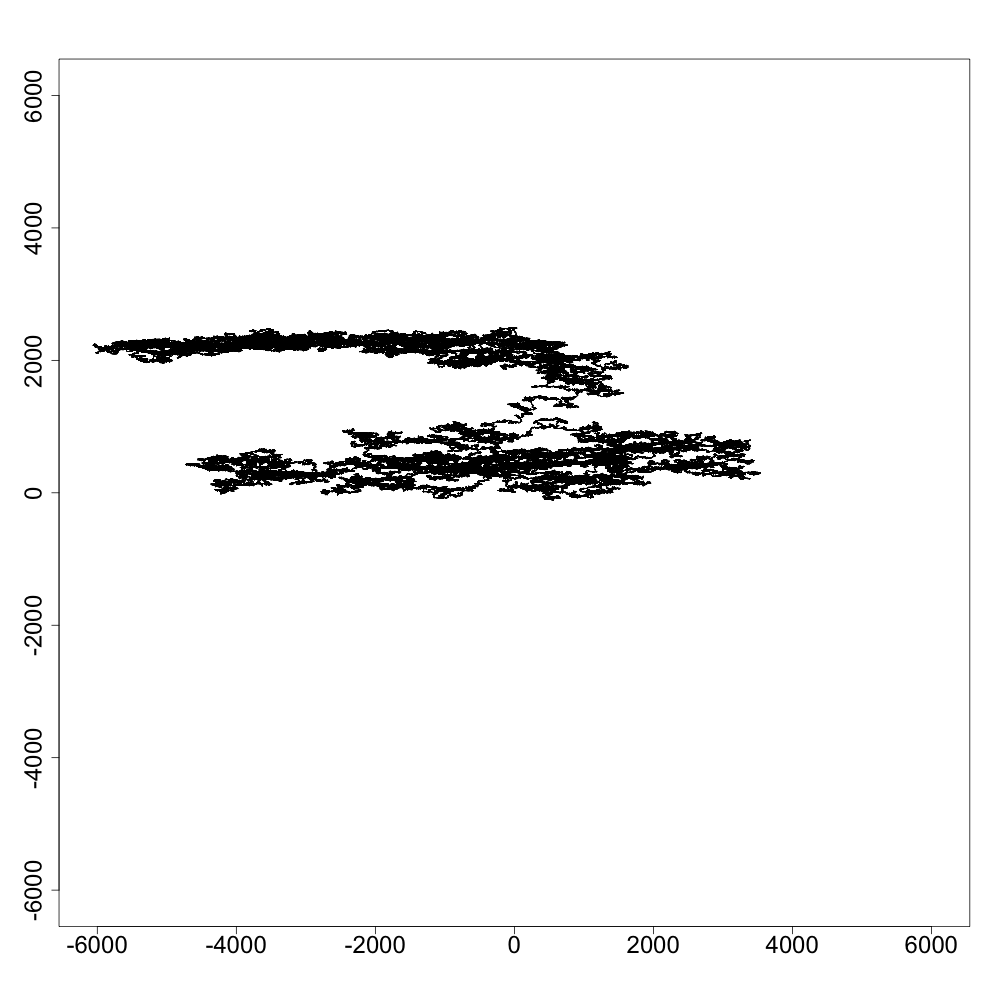}
\end{center}
\caption{A simulation of the path of the 2-dimensional balanced excited random walk, up to time $n=10^8$ with common horizontal and vertical scale.}
\label{fig:common_scale}
\end{figure}

\begin{figure}
\begin{center}
\includegraphics[scale=0.25]{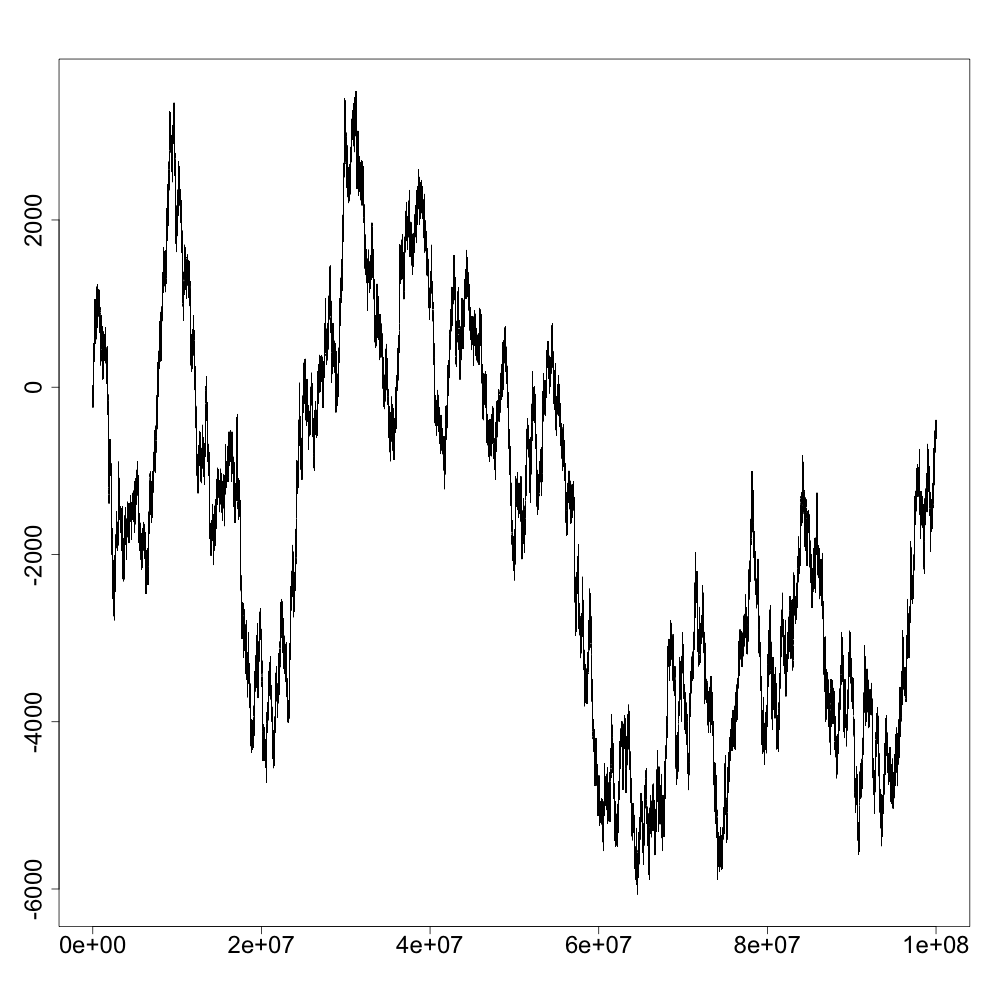}

\includegraphics[scale=0.25]{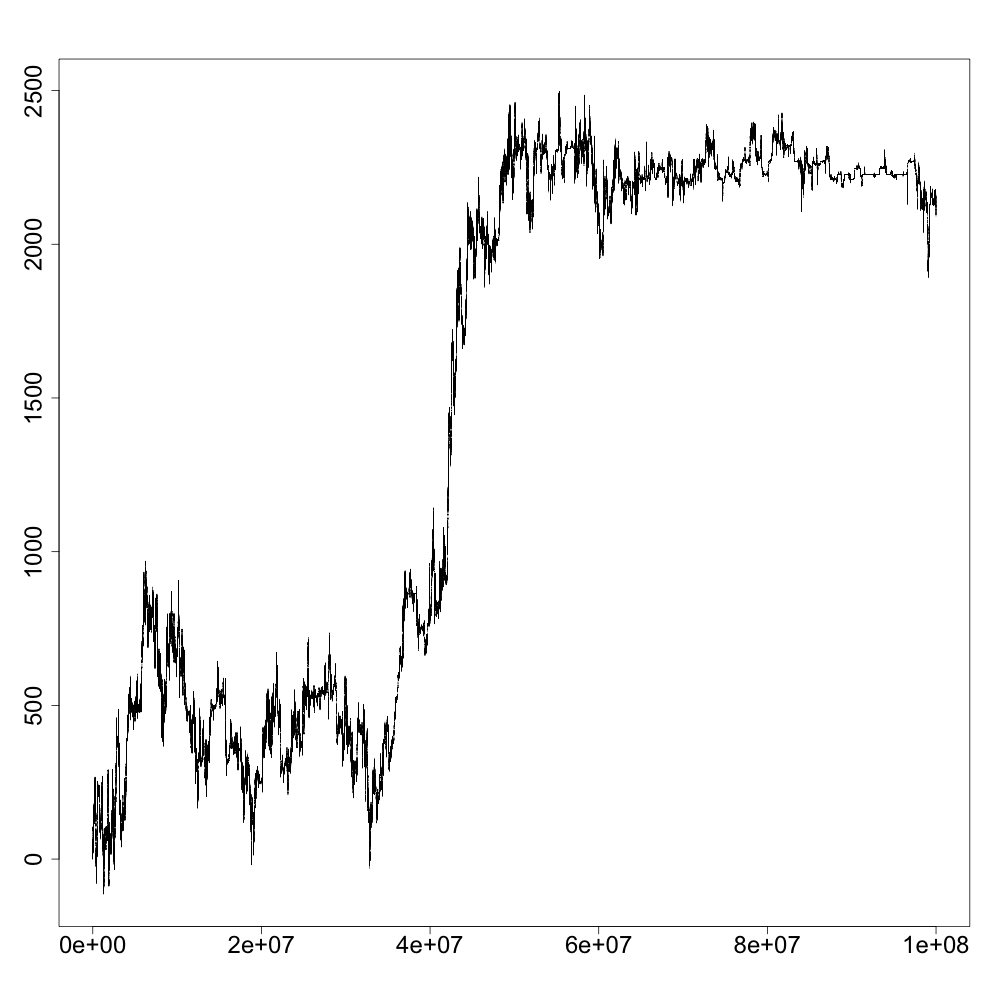}
\end{center}
\caption{A simulation of the horizontal (top) and vertical (bottom) displacement over time, up to time $n=10^8$.}
\label{fig:components}
\end{figure}

\subsection*{Organisation}
The remainder of this paper is organised as follows.  In Section \ref{sec:Abelian} we introduce an important tool for our analysis, namely that of stacks of instructions (one stack per site) that represent the moves that a walker would take from each site.  We show that multiple particles moving according to such a set of instructions satisfy a certain Abelian property -  the set of instructions used by the particles does not depend on the order in which the particles are moved.  In Section \ref{sec:recurrence} we use this Abelian property to verify Theorem \ref{thm:dichotomy}.  In Section \ref{sec:upper} we prove Theorem \ref{thm:range}(i).  In Section \ref{sec:LB} we prove Theorem \ref{thm:range}(ii).  The proof of Corollary \ref{cor:BM} is left as an exercise using Theorem \ref{thm:range}(i) and the approach indicated prior to that theorem.

\section{Abelian Property}
\label{sec:Abelian}
In what follows for every $m\ge 1$, we define the set $[m]=\{1,2,\ldots,m\}$.
A set of $\Z^2$ instructions is a collection $\bs{I}=(I(x,k))_{x \in \Z^2,k\in \N}$ where each $I(x,k)$ is either a unit coordinate vector $e\in \{\pm e_i:i \in \{1,2\}\}$ or a unit vector together with an instruction to cease movement.

Let $\bs{P}=(P_i)_{i \in J}$ be particles started at locations $(x_i)_{i \in J}$ (not necessarily unique), where $J$ is finite or countable.  A movement list $L=(L_i)_{i \in \N}$ for $\bs{P}$ is a sequence with each $L_i\in J$ such that every $j\in J$ appears infinitely often in the list. A movement list $L$ together with a set of instructions $\bs{I}$ defines a collection of paths associated to the particles as follows:

The first move is made by particle $P_{L_1}$.  It follows the instruction $I(x_{L_1},1)$.  Having made the first $k$ (attempted) moves of the particles, we next invite particle $P_{L_{k+1}}$ to move.  If it has not already been ceased by an instruction that it previously used, it follows the first instruction at its current location that has not already been followed.  If it has already been ceased then no instruction is used.  If a particle uses an instruction that consists of a unit vector together with a cease statement, then the particle first takes the step indicated by the unit vector and then ceases.

Let $\bs{I}(L)$ denote the collection consisting of the  instructions used by $L$.

\begin{prop}
  \label{prop:abelian1}
  Let $\bs{I}$ be a set of $\Z^2$ instructions and let $L$ and $L'$ be two movement lists for $\bs{P}$.  Then $\bs{I}(L)=\bs{I}(L')$.
\end{prop}

\begin{proof}
Suppose not.  Wolog $\bs{I}(L) \nsubseteq \bs{I}(L')$, and we let $I(x,k)$ be the first (in chronological order) instruction used by $L$ that is never used by $L'$.  Then $I(x,k)\ne I(x_i,1)$ for any $i$ since the first instruction at each $x_i$ is automatically followed for any $L$.  Let $(I(y_i,j_i))_{i \in A}$ denote the set of instructions used by $L$ up to but not including the instruction $I(x,k)$.  This includes all non-ceasing  instructions used by $L$ to move from neighbours of $x$ to $x$, up to but not including the time of the $k$th departure from $x$ for $L$.  Since these instructions all appear in $L'$ it follows that $L'$ must also have non-stopped particles visit $x$ at least $k$ times, so $I(x,k)$ must also be used by $L'$, which gives a contradiction.
\end{proof}

\begin{lemma}
\label{lem:mono0}
Suppose that $L$ is a movement list for a set of particles $\bs{P}$, and that $L'$ is a movement list for a set of particles $\bs{P}'\supset \bs{P}$ (i.e.~add some extra particles).  Then $\bs{I}(L)\subset \bs{I}(L')$.
\end{lemma}
\begin{proof}
Let $I(x,n)$ be an instruction used by $L$.  Then it is used at some finite position $j$ in that list.  By Proposition \ref{prop:abelian1} we can realise $\bs{I}(L')$ by moving the particles in $\bs{P}'$ in any order.  So consider the ordering that first uses $L_1,L_2,\dots, L_j$ and then continues in any order thereafter.  The instruction $I(x,n)$ is clearly used by this list.
\end{proof}
\blank{
Given a set of particles $\bp$, a set of instructions $\bs{I}$ and a movement list $L$ for $\bp$, let $M_{\bp}(x,L,\bs{I})$ denote the number of instructions ever used at $x$ by these particles.  It is immediate from Lemma \ref{lem:mono0} that 
\begin{align}
M_{\bp}(x,L,\bs{I}) \quad \text{ does not depend on $L$}. \label{rem:not_depend}
\end{align}   }
Consider a set $\bs{\Lambda}=(\Lambda_z)_{z \in \Z^2}$ of independent rate-one Poisson clocks.  A sequence of firings $(V_i,T_i)_{i \in I}$ where $I=\N$ or $I=[n]$,  and for each $i$, $(V_i,T_i) \in \Z^2\times \R_+$, $T_{i+1}\le T_i$ and $|V_{i+1}-V_i|=1$, is called a \emph{descending chain}.  If $I=[n]$ then the (finite) sequence is called a \emph{descending chain of length $n$}.   Let $B(z,r)$ denote the ball of radius $r$ (graph distance) centred at $z$.   
\begin{lemma}
\label{lem:chains}
For each $z\in \Z^2$ and $t\ge 0$ there exists $R=R(z,t)$ such that there are no infinite or finite descending chains $(V_i,T_i)_{i \in [n]}$ with $n \in \N$ such that  $V_1\sim z$, $T_1\le t$ and $V_n\notin B(z,R)$.
\end{lemma}
\begin{proof}
Fix $z_1\sim z$ and suppose that there is such a descending chain, then there is one (for some $n'\le n$) where all $V_i$ are distinct, and therefore the chain itself is of length at least $R-1$.
The proof of \cite[Lemma 4]{compass} shows that the expected number of descending chains of length $n$ started from $z$ with all vertices distinct is at most $(4t)^n/(n!)$.  Since any chain of length $k$ contains a chain of length $\ell<k$ the probability that there exists a descending chain of length $\ge n$, starting from $(z,t)$ is at most $(4t)^n/(n!)$.  Thus letting $A_n$ be the event that there is such a descending chain of length $n$ we have 
\begin{align}
\P(\cap_{m=1}^\infty \cup_{n=m}^\infty A_n)\le \P( \cup_{n=m}^\infty A_n)=\P(A_m)\to 0.
\end{align}
In other words $\P(\cup_{r=1}^\infty \cap_{n=r}^\infty A_n^c)=1$, which completes the proof.
\end{proof}

\section{Proof of Theorem \ref{thm:dichotomy}}
\label{sec:recurrence}
We begin by considering a collection of particles indexed by the sites of $\Z\times \{0\}$ (one particle per site), each walking in continuous time in the lattice $\Z^2$.  

Consider a probability space $(\Omega, \mc{F},\P)$ on which $(U(y,n))_{y\in \Z^2,n \in \N}$ are i.i.d.~$\sim U(0,1]$ random variables, $(\Lambda_y)_{y\in \Z^2}$ are i.i.d.~rate 1 Poisson processes, and $\bs{I}=(I(y,n))_{n\in \N, y \in \Z^2}$ are i.i.d.~(over $y$) variables that are Rademacher vertical steps for $n=1$ and Rademacher horizontal steps for $n>1$.  Also on this space let $(W^x_n)_{n \in \N,x \in \Z}$ be i.i.d.~Rademacher random variables.  All of the above random variables are mutually independent.

We have a set of particles $P$ indexed by $\Z$, with particle $x$ starting at position $(x,0)$ at time 0.  For each $x\in \Z$ we also have an associated ordering $\ord{x}$ of $\Z$, defined by 
\[x\ord{x} x+1\ord{x} x-1\ord{x} x+2\ord{x} x-2\dots.\]
The particles move in continuous time according to the following rules:

\medskip

{\bf \noindent Version A:}  
When the clock at $y \in \Z^2$ fires for the $n$-th time at some time $t$, we choose a particle from among those at $y=(u,v)$ at time $t-$ (if there are any, if not we do nothing) to read the next instruction, using $U(y,n)$ and the ordering $\ord{u}$.  To be precise, if the particles at $y$ at this time are $x_1\ord{u} \dots \ord{u} x_m$ then we choose particle $x_i$ if $U(y,n)\in ((i-1)/m, i/m]$.  This is well defined since there are no infinite descending chains (and hence there are only finitely many particles at any given site at any time, a.s.)  Every particle stops as soon as it returns to $\Z\times \{0\}$.  This system is invariant with respect to horizontal shifts, and ergodic since the environment (instructions and Poisson processes) is i.i.d.~over the sites.

\medskip

{\bf \noindent Version B:}  As for Version A, except that now on the $k$-th occasion when the $x$-particle is chosen to read an instruction of the form $I(y,1)$ for some $y \in \Z^2$, it instead takes a vertical step with sign determined by $W^x_k$.  This version is also stationary (horizontal shift) and ergodic.  The resulting collection of walks has the same law as in Version A.

\begin{lemma}
\label{lem:AB}
For the above collection of particles (Version A, or Version B), almost surely:
\begin{itemize}
\item[(a)]  Every particle returns to $\Z \times \{0\}$,
\item[(b)] Every vertex in $\Z^2 \setminus \{\Z\times \{0\}\}$ is visited infinitely often by the collection of particles.
\end{itemize}
\end{lemma}
\begin{proof}
It is sufficient to prove the result for version $B$, since the joint law of the particles in version $A$ is the same.  Therefore we proceed with Version B.

The $x$-particle will a.s.~only use finitely many of the $W^x_i$, but
we will define $S^x_0=0$ and $S^x_n=\sum_{k=1}^n W^{x}_k$ for all $n
\in \N$ anyway.   Let $\tau^x=\inf\{n\in
2\N:N^x_{n,+}=N^x_{n,-}\}=\inf\{n \in 2\N:S^x_n=0\}$, where
$N^x_{n,+}$ denotes the number of $+1$-valued random variables among
$\{W^x_1,\dots, W^x_n\}$ and $N^x_{n,-}=n-N^x_{n,+}$.  Then $\tau^x<\infty$ almost surely.  Let $\tau^{*x}$ denote the number of vertical moves actually made by the $x$-particle.  Then $1\le \tau^{*x}\le \tau^x$ almost surely (it is a priori possible that the movements of other particles prevent this particle from returning to $\Z\times \{0\}$, in which case $\tau^{*x}<\tau$).  
  Note that $\tau^x$ is $\sigma(W^x)$-measurable (where $W_x=(W^x_n)_{n \in \N}$) but $\tau^{*x}$ is not (it depends on the moves of other particles).

For $i,n \in \N$, let $L^x_{n,i}=\#\{k<
n:N^x_{k,+}=N^x_{k,-}+i\}=\#\{k<n:S^x_k=i\}$, which is the number of
vertical moves from level $i$ appearing in the list of moves
$(W^x_1,\dots, W^x_n)$.  Then $\E[L^x_{\tau^x,i}]=1$ for each $i$,
since $L^x_{\tau^x,i}$ is the number of visits to (departures from)
$i$ by a simple random walk excursion from 0.  These quantities are
$\sigma(W^x)$ measurable.  On the event $J^x_i$ that the particle at $(x,0)$ gets stuck at level $i$ or above (this event is not $\sigma(W^x)$ measurable) we have $L^x_{\tau^{*x},i}\le L^x_{\tau^x,i}-1$.  This shows that if $\P(J^x_i)>0$ then $\E[L^x_{\tau^{*x},i}]<1$.

For $z\in \Z^2$, let $p_z$ be the probability that $z$ is hit by the family of continuous-time excursions.  By (horizontal)  stationarity, $p_{(k,i)}=p_{(0,i)}$ for every $i$.  Let $K^x_{i}$ be the event that the walk at site $(x,0)$  gets stuck on level $i$.  If $\P(K^0_{i})>0$ then by ergodicity a.s.~there exist (infinitely many) particles that get stuck on level $i$, and then almost surely every vertex in level $i$ is visited (infinitely often).   Indeed, let  $t$ be the unit shift in the horizontal
  direction acting on $\Omega$, so that $t(\omega)=(\omega_{z+(1,0)})_{z\in\mathbb{Z}^2}$.
  Note that  $\{t^n:n\in\mathbb Z\}$ is an ergodic group
  of transformations acting on the space $(\Omega,\mathcal B(\Omega), \mathbb P)$.
  Therefore, if $N_{n,i}$ is the  number of particles within the sites
  $(-n,0), (-n+1,0),\ldots, (n,0)$ that get stuck at level $i$, we have that
a.s.
  $$
\lim_{n\to\infty}\frac{1}{2n+1}N_{n,i}=\P(K^0_{i}),
$$
which implies that the total number $N_i$ of particles that get stuck at level $i$, then 
$N_i=\lim_{n\to\infty}N_{n,i}=\infty$. In particular, if $\P(K^0_{i})>0$ then $p_{(k,i)}=1$ for each $k\in \Z$.  On the other hand, if $\P(K^0_{i})>0$ then we know that $\E[L^0_{\tau^{*0},i}]<1-\vep_i$, and again by ergodicity we have that the proportion $p_i$ of vertices at level $i$
that are visited is smaller than $1-\vep_i$: 
  $$
p_i=\lim_{n\to\infty}\frac{1}{2n+1}\sum_{x=-n}^n L_{\tau^{*x},i}^{x}=\E[L^0_{\tau^{*0},i}]<1-\vep_i.
$$
 This gives a contradiction.  Therefore $\P(K^0_i)=0$ for every $i\in \N$, and we conclude that every particle almost surely returns to level 0, or in other words, $\tau^{*x}=\tau^x$ almost surely.   This proves (a).

  Now note that we have from (a) that for each $i\in \N$ and $x \in \Z$, $\E[L^x_{\tau^{x},i}]=1$. We claim
    that this implies that $p_{(0,i)}=1$. To show this, define for each
    $z\in\mathbb Z^2$,
    $\mathcal L_z$ as the number of vertical moves made (by any of the particles) from site $z$.
    Note that since $\mathcal L_z=0$ when site $z$ is never visited and
    $\mathcal L_z=1$ whenever site $z$ is visited, we have that

    $$
\E[\mathcal L_z]=p_z.
$$
By ergodicity the density of vertical moves per unit length in level $i\in \N$ is
$$
\lim_{n\to\infty}\frac{1}{2n+1}\sum_{j=-n}^n \mathcal L_{(j,i)}=p_{(0,i)}.
$$
The same density can be computed again by ergodicity as
$$
\lim_{n\to\infty}\frac{1}{2n+1}\sum_{x=-n}^n  L^x_{\tau^{x},i}=\E[\mathcal L^0_{\tau^{0},i}]=1.
$$
It follows that $p_{(0,i)}=1$.

    This proves that every vertex is visited by the collection of excursions.  We now show that every vertex in level $i\in \N$  is in fact visited infinitely often by this family of excursions.  Fix $i$, and let $\hat a_n$ 
    be the number of times level $i$ is entered from above or below at $(n,i)$.  Since $(n,i-1)$ is visited almost surely, and $(n,i+1)$ is visited almost surely, and each site has an i.i.d.~$\pm$ vertical step as the first move in the stack at that site, the numbers $(\hat a_n)_{n \in \Z}$ are i.i.d. $\Bin(2,1/2)$ random variables.

The first visit to $(n,i)$ results in a vertical exit from that site.  All subsequent visits result in a horizontal departure.  Let $h_{n}$ denote the number of horizontal departures from $(n,i)$ and $T_n$ denote the total number of visits to $(n,i)$.  Then we have the relationship
\[h_n=T_n-1.\]
We wish to show that $T_n=\infty$ for each $n$ (a.s.).  If $T_n=\infty$ for some $n$ with positive probability then by ergodicity $T_n=\infty$ for some $n$ and therefore $T_k=\infty$ for all $k$ (since all but the first move from $(n,i)$ is horizontal).  We shall therefore suppose that $T_n<\infty$  for every $n$ and look for a contradiction.
Now the number $T_n$ of visits to $(n,i)$ consists of the number $a_n$ of vertical visits to $(n,i)$ plus the number of visits from $(n-1,i)$ (call this $\ell_{n,+}$), plus the number of visits from $(n+1,i)$ (call this $r_{n,-}$).  Similarly, the number $h_n$ of horizontal departures from $(n,i)$ is the number of departures from $(n,i)$ to $(n-1,i)$ 
(call this $\ell_{n,-}$), plus the number of departures from $(n,i)$ to $(n+1,i)$ (call this $r_{n,+}$).  It follows that we have
\[\ell_{n,-}+r_{n,+}=\hat a_n+\ell_{n,+}+r_{n,-}-1.\]
If $T_n$ is finite then all of these quantities are finite and we can write
\[\hat a_n-1=\ell_{n,-}+r_{n,+}-\ell_{n,+}-r_{n,-}.\]
  Let $m_n$ be the number of signed crossings from $(n,i)$ to $(n+1,i)$ (with reverse crossings counted negatively).  Then $m_n$ is an ergodic process, and $m_n=r_{n,+}-r_{n,-}$ and $m_{n-1}=\ell_{n,+}-\ell_{n,-}$.  Thus we have
 \[\hat a_n-1=m_n-m_{n-1}.\]
 Summing over $n$ from $1$ to $t$ we get
 \[\sum_{n=1}^t(\hat a_n-1)=m_t-m_0.\]
 By assumption the $m_n$ are a.s. finite, so there is some $K$ so that $|m_n|\le K$ with positive probability and by ergodicity the set $S = \{n : |m_n| \le K\}$ has positive density.  However, $\sum_{n=1}^{t}(\hat a_n-1)$ is a (lazy) simple random walk (recall that $(\hat a_n)_{n \in \Z}$ are i.i.d.~Bin$(2,1/2)$), and so it is impossible for it to be bounded by $K+|m_0|$ at all points of a set with positive density. This gives a contradiction, thus $T_n$ must be infinite, a.s.
\end{proof}

\subsection{Connection with the BERW}
Let $D_y=\{y \text{ is visited infinitely often by the walk}\}$.
\begin{lemma}
\label{lem:lev0}
For the balanced BERW in 2 dimensions, a.s.
\[D_{(u,0)} \text{ for some }u \in \Z \Rightarrow D_y \text{ for every }y \in \Z^2.\]
\end{lemma}
\begin{proof}
It is sufficient to show that if $\P(D_{(u,0)})>0$ then almost surely on $D_{(u,0)}$ we have that every vertex is visited infinitely often.   In order to show this, note first that if $D_{(u,0)}$ occurs then a.s.~$D_{(x,0)}$ occurs for every $x\in \Z$ since e.g.~on every visit to $(u,0)$ (except the first) there is probability $1/2$ of then visiting $(u-1,0)$ (so $(u-1,0)$ is also visited i.o.).  Therefore it remains to show that every vertex in every other level is visited i.o.  To show this, it is sufficient to prove that for any $y$ in level $i\ne 0$, and any $n \in \N$, $y$ is visited at least $n$ times a.s.  Note that the BERW can be realised as the movement of a particle (started at the origin) in an i.i.d.~instruction environment $\bs{I}$ as used in Version A.  

So let $\bs{I}$ be given, and suppose that $D_{(u,0)}$ occurs (and hence that $D_{(x,0)}$ occurs for every $x\in \Z$) for the BERW using the instruction environment $\bs{I}$.  Fix $y$ and $n$.  We know that $y$ is visited i.o. a.s.~by Version A particles (using the same $\bs{I}$), hence there is a.s. some finite time $T_0$ at which $y$ has been visited  $n$ times by time $T_0$.  These visits occurred by particles using some instructions $\bs{I}_0\subset \bs{I}$ prior to time $T_0$.  By Lemma \ref{lem:chains} $\bs{I}_0$ is finite and this set of instructions was used by a finite set of particles $\bs{P}_0=\{P_1,\dots, P_k\}$ with starting locations $\{x_1,\dots, x_k\}\ni 0$ for some $k$. 

Since the BERW in $\bs{I}$ visits every vertex in $\Z \times \{0\}$ infinitely often as above, 
there is some smallest time $N_0<\infty$ at which the BERW has departed from each of the sites $x_1,\dots, x_k$ above and returned to level $0$.  We will show that by time $N_0$ the BERW has used every instruction in $\bs{I}_0$, which implies that it has also visited $y$ at least $n$ times.

To achieve this, we will define a movement list $L$ for the particles $\bs{P}_0$ that will match the behaviour of the BERW until time $N_0$:  let $\rho_0=0$, and $v_0=0$.  For $j=1, \dots, k-1$ let $\rho_j=\inf\{n>\rho_{j-1}: X_n\in \{x_1,\dots, x_k\}\setminus \{v_0,\dots, v_{j-1}\}\}$ be the first time that the BERW has visited $j$ distinct elements of $x_1,\dots,x_k$, and $v_j$ be such that $(v_j,0)=X_{\rho_j}$.  Let $\rho_{k}=\inf\{n>\rho_{k-1}:X_n\in \Z\times \{0\}\}$.  Note that $\rho_k=N_0$.

The first $\rho_{k}$ of the elements of the list $L$ are defined so that the $r$-th element of $L$ is $v_j$ if and only if $\rho_{j}<r\le \rho_{j+1}$.  Thus, the first $\rho_1$ of the elements of $L$ are $0$ etc.  Subsequent elements of $L$ just cycle through the particles in $\bs{P}_0$ (one element each) thereafter.  Modify the instructions $\bs{I}$ to get a new instruction set $\bs{I}^*$ with the property that any instruction $I(y,n)$ used at time $\rho_j-1$ for some $j\ge 1$ is replaced with the same instruction plus a cease instruction (all other instructions remain the same).  In other words, we have constructed a particle system with particles $\bs{P}_0$, a movement list for those particles, and an instruction set such that the sequence of particle movements in this construction exactly matches the original BERW until time $N_0$ (after this time we don't care what happens).   Another way of phrasing this is that the BERW walks as particle $v_0=0$ until time $\rho_1$, when it continues to walk as particle $v_1$, until time $\rho_2$ etc.  

Now note that the Version A system of excursions remains unchanged if we replace the instruction set $\bs{I}$ with the modified $\bs{I}^*$.  This is because instructions $I((x,0),1)$ remain unchanged $I((x,0),1)=I^*((x,0),1)$ for every $x \in \Z$, and each excursion in Version A was stopped upon re-entry into $\Z \times \{0\}$.  

Now define a movement list $L'$ for the particles $\bs{P}_0$ for Version A  with instruction set $\bs{I}^*$ according to the order of movement of these particles in Version $A$ up to time $T_0$, followed by cycling through $\bs{P}_0$ (one element each) thereafter.  Thus, $\bs{I}^*_0=\bs{I}_0\subset \bs{I}^*(L')$.  This $\bs{P}_0$-particle process defined from Version A has moves that are identical to the moves of the particles $\bs{P}_0$ of Version $A$ until at least time $t$, so in particular this particle process visits $y$ at least $n$ times up to time $t$.

By Proposition \ref{prop:abelian1}, $\bs{I}^*(L')=\bs{I}^*(L)$.  Thus the particles in the BERW particle process use all of the instructions in $\bs{I}^*_0$.  Since each particle has been stopped by time $N_0$, this means that the BERW uses every instruction in $\bs{I}_0$ by time $N_0$ (and hence visits $y$ at least $n$ times by time $N_0$), as required.
\end{proof}

\begin{proof}[Proof of Theorem \ref{thm:dichotomy}]
In view of Lemma \ref{lem:lev0} it is sufficient to show that for any level $i \ne 0$, on the event $D_z$ that $z$ (in level $i$) is visited infinitely often by the BERW, almost surely every site is visited infinitely often.  Since the proof is similar to that of Lemma \ref{lem:lev0} we will be brief.  

Firstly note that it is trivially true that on the event $D_z$, almost
surely every 
site  in the same level as $z$ is visited infinitely often by the BERW.  Therefore it is sufficient to show that for any $y$ not in level $i$, $y$ is visited at least $n$ times almost surely (on the event $D_z$).

Now consider Version A, but with particles started from points in $\Z\times \{i\}$.  These particles visit every vertex outside level $i$ infinitely often, and hence there is some finite time $T_0$ by which they have visited $y$ at least $n$ times.  These visits used some set of instructions $\bs{I}_0$ that is finite, and these instructions were used by a finite set of particles, say $\bs{P}_0=\{P_1,\dots, P_k\}$, started at points  $\{(x_1,i), \dots, (x_k,i)\}$.  

Run the BERW with the same set of instructions, and note that by assumption each $(x_j,i)$ for $j=1,\dots, k$ is eventually hit by the BERW, and after departing from each of these points at least once, it will return to level $i$ at some finite time $N_0$.

Define a movement list $L$ for the set of particles $\bs{P}_1=\bs{P}_0\cup \{P_0\}$ where $P_0$ is started at $\{(0,0)\}$ by modifying instructions (to get new instructions $\bs{I}^*$) used by the BERW to hit each starting point (of a particle in $\bs{P}_0$) for the first time (modify such an instruction $I(w,n)$ to do the same move but cease the particle when it makes this move and reaches the relevant starting point of particles in $\bs{P}_0$), then start moving the particle just reached etc.  This defines the first $N_0$ moves in the movement list.  Now cycle between the particles $\bs{P}_1$ in some fixed order to complete the movements list $L$. Let $L^*$ be the movement list $L$ with all occurrences of $P_0$ deleted (so $L^*$ is a movements list for $\bs{P}_0$, rather than $\bs{P}_1$).  

Similarly, define a movement list $L'$ from Version A with the particles $\bs{P}_0$ as usual until time $T_0$, and then extend to an infinite movement list by cycling through $\bs{P}_0$.  
Modifying the instructions from $\bs{I}$ to $\bs{I}^*$ as above does not affect the behaviour of these particles up to time $T_0$, since they are each stopped when they enter level $i$ from outside.

Then we have by Proposition \ref{prop:abelian1} and Lemma \ref{lem:mono0} that $\bs{I}_0=\bs{I}^*_0\subset \bs{I}^*(L')=\bs{I}^*(L^*)\subset \bs{I}^*(L)$, but since all of the particles are stopped by time $N_0$ for the BERW construction, this means that all instructions in $\bs{I}^*_0$ have been used in the BERW up to time $N_0$, so all instructions in $\bs{I}_0$ have been used by the BERW up to time $N_0$.  Thus also $y$ has been visited at least $n$ times by the BERW up to time $N_0$.
\end{proof}

\section{Upper bound on the range}
\label{sec:upper}

In this Section we prove the non-trivial upper bound on the range given in \cref{thm:range}(i). 

We will actually consider a more general problem, in which we construct a 2-dimensional walk $\bs{Z}=(Z_t)_{t\ge 0}$ from a pair of independent one-dimensional simple random walks, say $\bs{X}=(X_n)_{n \ge 0}$ and $\bs{Y}=(Y_n)_{n \ge 0}$.
The horizontal steps of $\bs{Z}$ are given exactly by the steps of $\bs{X}$ and the vertical steps of $\bs{Z}$ are given exactly by the steps of $\bs{Y}$.
What we are allowed to choose at each time $n \in \N=\{0,1,2,\dots\}$ is whether to take the next horizontal step or the next vertical step.
For example, we use a vertical step if $\bs{Z}$ is at a new vertex and a horizontal step otherwise, this process is exactly the BERW.
Our objective is to maximize the size of the range (i.e.~trace) of the resulting 2-dimensional walk $\bs{Z}$.

If we were to only ever use horizontal (resp.~vertical) steps, the range of the two-dimensional random walk $\bs Z$ would just be the range of the 1-dimensional walk $\bs{X}$ (resp.~$\bs{Y}$), which is typically of order $\sqrt{t}$ at time $t$ (with some exceptional times when the range is slightly larger, as given by the law of the iterated logarithm).
It is possible to get a range of order $\sqrt{t}$ also if we must use both directions, for example by making steps in only one of the coordinates in blocks of quickly growing lengths.
On the other hand if we choose to take a horizontal or vertical step
based on an independent sequence of coin tosses then the random walk
$\bs{Z}$ will just be a standard simple two-dimensional random walk,
with a range whose cardinality grows like  $\pi t/\log t$,  at time $t$ (see for example \cite{DE51}).
The main question is whether knowledge of the steps of the two walks in part or in their entirety allows for a significantly larger range.

Let us state a more precise version of this question. 
We call a sequence $\bs{x}=(x_n)_{n\ge 0}$ on $\Z$, a  \textbf{nearest neighbour path} if $|x_{n+1}-x_n|=1$ for all $n\ge 0$.
Similarly, we call a sequence $\bs{z}=(z_n^{\sss [1]},z^{\sss[2]}_n)_{n\ge 0}$ in $\Z^2$ a \textbf{nearest neighbor path} if $\|(z^{\sss[1]}_{n+1},z^{\sss[2]}_{n+1})-(z^{\sss[1]}_n,z^{\sss[2]}_n)\|_1=1$ for all $n\ge 0$, where $\|\cdot\|_1$ denotes the $l^1$ norm.
A \textbf{timing sequence} is a monotone increasing (north/east) nearest neighbour sequence in $\Z^2$ starting at $(0,0)$.
More explicitly, these are sequences $\bs{q}=(n_t,m_t)_{t\in \Z_+}$ with values in $\Z_+\times\Z_+$, with $n_0=m_0=0$ and such that for each $t\ge 0$ we have that $(n_{t+1},m_{t+1})$ is either equal to $(n_t+1,m_t)$ or to  $(n_t,m_t+1)$.
Note that $n_t+m_t=t$ for any timing sequence.  
We denote the set of timing sequences by $\cS$, and the set of restrictions of timing sequences to the first $t$ elements $\cS_t$.
Two infinite nearest neighbor paths $\bs{x}$ and $\bs{y}$ on $\Z$, and a timing sequence $\bs{q}\in\cS$ give rise to a nearest neighbor path $\bs{z}(\bs{x},\bs{y},\bs{q})=(z_t)_{t\ge 0}$ on $\Z^2$ defined by $z_t=(x_{n_t},y_{m_t})_{t\ge 0}$. 

Now consider two independent simple random walk paths $\bs{X}$ and $\bs{Y}$, and an independent source of randomness $\bs{U}$
on some probability space $(\Omega,\F,\P)$.  We write $\F^{\bs{X}}$ (resp. $\F^{\bs{Y}}$, $\F^{\bs{U}}$) for the $\sigma$-algebra generated by the entire sequence $\bs{X}$ (resp.~$\bs{Y}$, $\bs{U}$).  
The natural filtrations of $\bs{X}$ and $\bs{Y}$ are denoted by $(\F^{\bs{X}}_n)_{n\ge 0}$ and $(\F^{\bs{Y}}_n)_{n\ge 0}$ respectively. 

A \textbf{timing rule} is any element $\bs{Q}=(N_t,M_t)_{t \in \Z_+}\in\cS$, measurable with respect to $\F^{\bs{X}} \otimes \F^{\bs{Y}} \otimes \F^{\bs{U}}$.
   
 
Each triple $\bs{X},\bs{Y},\bs{Q}$ defines a random nearest neighbor path $\bs{Z}=(Z_t)_{t\ge 0}=\bs{z}(\bs{X},\bs{Y},\bs{Q})$ satisfying $Z_t=(X_{N_t},Y_{M_t})$.
We call this random walk the \textbf{two-dimensional random walk
  generated by $(\bs{X}$, $\bs{Y})$, and the rule $\bs Q$}.
Let $\mc{R}_t^{\sss{\bs{Z}}}=\{Z_0,Z_1,\dots, Z_{t-1}\}$ be its range up to time $t$, and $R_t^{\sss {\bs{Z}}}=\#\mc{R}_t^{\sss {\bs{Z}}}$.

\blank{{\color{blue} Let $\mathcal G_s$ be the smallest $\sigma$-field
  generated
  by the sets $\{N_s=k\}\cap\{M_s=s-k\}\cap N\cap M\cap A\cap B$, with
   $A\in\mathcal F^{\bs 
     X}_k$, $B\in\mathcal F^{\bs Y}_{s-k}$, $k\ge 0$,
   and $N=\{N_1=k_1,\ldots, N_{s-1}=k_{s-1}\}$,
   $M=\{M_1=k'_1,\ldots, M_{s-1}=k'_{s-1}\}$, with $0\le
   k_1\le\cdots\le k_{s-1}<k$
   and $0\le k'_1\le\cdots\le k'_{s-1}<s-k$.
}}

\begin{open}
\label{open:rules}
  Is there a timing rule such that with positive $\P$-probability 
  \[
    \lim_{t\to\infty}\frac{R^{\sss{\bs{Z}}}_t}{t/\log t}=\infty?
  \]
  If yes, how close to $t$ can the range $R_t^{\bs{Z}}$ be?
\end{open}


\medskip

We study an easier variant of the above question, which restricts to rules that do not depend on the entire random walk trajectories, and in particular includes the BERW.
We say that a rule $\bs{Q}$ is \textbf{$\bs{Y}$-adapted}, if the decision of the next move at time $t$ is determined by $(Z_i)_{i\le t}$, together with the entire walk $\bs{X}$ and the extra randomness $\bs{U}$, but does not depend on the steps of $\bs{Y}$ that have not yet been used.
Formally, let $\mc{G}_s$ be the $\sigma$-field generated by the sets 
\[ \{N_1=k_1,\dots, N_s=k_s\}\cap A \cap B\cap C,\]
for $0\le k_1\le \dots \le k_s\le s$, $A\in\F^{\bs{X}}$,  $B\in\F^{\bs{Y}}_{s-k_s}$, and $C\in \F^{\bs{U}}$.
We say that a rule $\bs{Q}=(N_t,M_t)_{t\ge 0}$ is $\bs{Y}$-adapted if for every $t\ge 0$ and $\bs{q}\in \cS_{t+1}$, the event $(Q_i)_{i\leq t} = \bs{q}$ is $\mc{G}_t$ measurable.

Let $\sigma_0=0$ and for $i \in \N$ let $\sigma_i=\inf\{t>\sigma_{i-1}:N_t>N_{t-1}\}$.  These are the jump times of the first coordinate of $\bs{Z}$, so this coordinate is constant on $[\sigma_i,\sigma_{i+1}-1]$.
For each $i\ge 1$, define the $\sigma$-algebra $\mathcal
  G_{\sigma_i}$
  as the smallest $\sigma$-algebra generated by the sets

  $$
  \{\sigma_i=k\}\cap A,
  $$
where $A\in\mathcal G_k$.

\begin{theorem}
  \label{upper-bound-theorem} 
  Let $\bs X$ and $\bs Y$ two independent one-dimensional simple random walks, independent of
  $\bs{U}$.
  Let $\bs Q$ be a timing rule that is $\bs Y$-adapted, and $\bs{Z}=\bs{Z}(\bs{X},\bs{Y},\bs{Q})$ be the associated two-dimensional walk.
  Suppose that there is a constant $K$ such that for all $i\ge 0$,
  \begin{equation}
    \label{eee}
    \E[\sigma_{i+1}-\sigma_i|\mathcal G_{\sigma_i}]\le K, \text{ a.s.}
  \end{equation}
  Then, 
  \[
    \limsup_{t\to\infty}\frac{\E[R^{\sss{\bs{Z}}}_t]}{t/\sqrt{\log\log t}}<\infty,
  \]
  and in particular $R^{\sss{\bs{Z}}}_t = o(t)$ in probability.
\end{theorem}

This theorem falls short of answering Open Problem  \ref{open:rules} in 3 ways:
First, it requires that the timing rule is $\bs{Y}$-adapted,
second, it requires that the horizontal (first coordinate) steps are sufficiently frequent,
and lastly, it only shows $\E[R^{\sss{\bs{Z}}}_t]$ grows no faster than $\frac{t}{\sqrt{\log\log t}}$.
By Markov's inequality, this theorem implies that for every $\vep>0$ there exists $K_\vep>0$ such that 
\begin{align}
  \limsup_{t\to\infty}\P\Big(\dfrac{R^{\sss{\bs{Z}}}_t}{t/\sqrt{\log\log t}}>K_\vep\Big)<\vep.
\end{align}

We claim that the conditions of \Maaak{\cref{upper-bound-theorem} }
apply to the balanced excited random walk.
To see this let $\bs{X}, \bs{Y}$ be independent one-dimensional simple random walks, and let $\bs{Q}$ be the timing rule that takes a step from $\bs{Y}$ on the first departure from any site in $\Z^2$ and otherwise a step from $\bs{X}$.
Now observe that any three consecutive steps include a horizontal step with probability at least 1/2, irrespective of the past, which implies that (\ref{eee}) is satisfied with $K=6$.
Indeed, if the first step in this interval is vertical then either the walker is now at a previously visited site (so must take a horizontal step) or it is at an unvisited site and has probability $1/2$ of stepping back to its previous location, whence the third step must be horizontal.
Hence, \cref{thm:range}(i) is an immediate Corollary of Theorem \ref{upper-bound-theorem}.

\subsection{Notation and preliminary lemmas}

We now introduce some important notations that will be used in the proof of 
\cref{upper-bound-theorem}, including the concept of slow intervals.
We furthermore prove here some preliminary lemmas.
Everything in this sub-section refers to a (nearest neighbour, symmetric) 1-dimensional discrete-time simple random walk $\bs{S}=(S_t)_{t\ge 0}$.
For a time interval $I=[s,s+t]$ (with $s,t\in \Z_+$) and $\vep>0$, we
say that  $\bs{S}$ is $\vep$-slow on $I$, or that $I$ is $\vep$-slow
for the walk $\bs{S}$, if the range $\{S_s,S_{s+1},\dots, S_{s+t}\}$
of the random walk during the time interval $I$ 
  is contained in some closed interval of length $\vep\sqrt{t}$.  For such an interval $I=[s,s+t]$, write $|I|=t$ for the length of the interval, and write $\mc{R}^{\bs{S}}_I$ for the range of $\bs{S}$ during the interval $I$ as above.  Let $R^{\bs{S}}_I$ denote the cardinality of $\mc{R}^{\bs{S}}_I$, and note that $R^{\bs{S}}_I\le |I|+1$.  Then $I$ is $\vep$-slow for $\bs{S}$ means that 
\begin{equation}
R^{\bs{S}}_I-1\le \vep\sqrt{|I|}.\label{eps_slow_def}
\end{equation}
Note that for $|I|\ge 2$ this can only occur if $t\ge\vep^{-2}$, since at least two points must be visited during such an interval. 

A key step in our argument is to show that for a simple random walk, the time interval $[0,n]$ can be almost covered by $\vep$-slow intervals.
Towards this we first estimate the probability that a given interval is $\vep$-slow.
Our first lemma is a fairly standard estimate of the probability that an interval is slow for a random walk.

\begin{lemma}
  \label{L:slow_pr}
  Let $\tau_m$ be the hitting time of $\pm \lfloor m\rfloor$ by a simple random walk started at $0$.
  Then there is a constant $c_1$ so that for any $m\ge 2$ and $k\geq 1$ we have $\P(\tau_m > km^2) \geq e^{-c_1k}$.
\end{lemma}

Note that $m,k$ need not be integers.

\begin{proof}
  By Donsker's invariance principle there is a constant $c$ so that for any $m\ge 2$ and any $x\in[-m/2,m/2]$,
  \[ \P_x\big(\tau_m \geq m^2 , |S_{m^2}|\leq m/2\big) \geq c. \]
  It follows by induction and the Markov property that for any integer $k$,
  \[ \P_x(\tau_m \geq km^2) \geq c^k. \]
  Finally, if $k\ge 1$ is not an integer, we have
  \[
    \P_x(\tau_m \geq km^2)
    \geq \P_x(\tau_m \geq \lceil k\rceil m^2)
    \geq c^{\lceil k\rceil}
    \geq c^{2k}.   \qedhere
  \]
\end{proof}

\begin{lemma}\label{L:more_slow}
For $\vep\in (0,1)$ and $i \in \N$, let  $A_i=A_i(\vep)$ be the event that $[0,2^i]$ is $\vep$-slow for $\bs{S}$.
   There exists a constant $c_3$ and independent events $(\tilde A_i)_{i\in \N}$ such that 
for every $i\in \N$ with $2^i > 2^8\vep^{-2}$
\[   \P(\tilde A_i) \geq e^{-c_3\vep^{-2}} \text{ and  }\tilde A_i\subset A_i.\]
\end{lemma}

\begin{proof}
  Consider the stopping times $(T_i)_{i\ge 0}$ defined by $T_{0} = 0$ and 
  \[
    T_i = \inf \big\{ t > T_{i-1} : |S_t-S_{T_{i-1}}| \geq \lfloor \vep 2^{i/2} \rfloor \big\},
  \]
  that is the time since the previous stopping time in the sequence it takes the walk to move $\lfloor \vep 2^{i/2} \rfloor$.
  By the Markov property, the increments $T_i-T_{i-1}$ are independent.
  Let $i_0 =i_0(\vep)= \lfloor \log_2 (4\vep^{-2}) \rfloor$.
  Define the events
  \[
    \tilde A_i(\vep) = \{T_i - T_{i-1} > 2^i\}.
  \]
  Applying \cref{L:slow_pr} to the walk after time $T_{i-1}$ (i.e. $S_{T_{i-1}+n}-S_{T_{i-1}}$) with $m=\vep 2^{i/2}$ and $k=\vep^{-2}$,
  we have for every $i> i_0$ that $\P(\tilde A_i(\vep)) \geq e^{-c_1/\eps^2}$.
  Now, for $t\le T_i$, we must have
  \[
    |S_t| \le \sum_{j=1}^i \floor{\vep 2^{j/2}} \le 4\vep 2^{i/2}.
  \]
    On the event $\tilde A_i(\vep)$ we have $2^i< T_i$, and therefore $[0,2^i]$ is $(8 \vep)$-slow for $\bs{S}$.  This shows that $\tilde{A}_i(\vep)\subset A_i(8\vep)$ for $i\ge i_0(\vep)$.  Therefore $\tilde A_i(\vep/8)\subset  A_i(\vep)$ and $\P(\tilde A_i(\vep/8)) \geq e^{-c_1/(\eps/8)^2}$  for $i\ge i_0(\vep/8)$ and the claim follows.    
\end{proof}

\medskip

A \textbf{dyadic interval} is an interval $I\subset\N$, of the form $[j 2^i,(j+1)2^i]$ for some $i,j\in\N$.
The following states that any point is highly likely to be contained in some slow dyadic interval.

\begin{prop}
  \label{P:slow_diad}

  There exists a constant $\delta$ such that for any $\vep\in (0,1)$, any $k \in \N$, and any $0\le s\le 2^k$ the probability that $s$ is not contained in any $\vep$-slow dyadic interval (for $\bs{S}$) in $[0,2^k]$ is at most
  \[
    \left(1-e^{-\frac{\delta}{\vep^2}} \right)^{k-9 -2|\log_2(\vep/4)|}.
  \]
\end{prop}

\begin{proof}
Fix $\vep\in (0,1)$ and let $k
>9+2|\log_2\vep|$.
  While the event we are interested in depends only on $(S_i)_{0\le i\le 2^k}$, it is convenient to extend the random walk $\bs{S}$ to a doubly infinite random walk, so that we can talk about intervals $[a,b]$ with $a<0$.
Let $i_0=i_0(\vep) = \lfloor \log_2 (2^8\vep^{-2}) \rfloor$, and consider the intervals $[s,s+2^i]$ for $i>i_0$.
  By \cref{L:more_slow} there are independent events $\tilde A_i$ (depending only on the steps of the walk after time $s$), with $\P(\tilde A_i) \geq e^{-c_3\vep^{-2}}$ such that on $\tilde A_i$ the interval $[s,s+2^i]$ is $\vep$-slow.
  Similarly, there are independent events $\tilde A'_i$ (depending only on the steps of the walk up to time $s$) with the same probability bound, on which the intervals $[s-2^i,s]$ are $\vep$-slow.  The collection $\tilde A_i$ is also independent of the collection $\tilde A'_i$ since they depend on disjoint sets of steps of the walk.
  
  Let $B$ be the event that there is no $i\in (i_0,k]$ so that $[s,s+2^i]$ and $[s-2^i,s]$ are both $\vep$-slow.
  Then 
  \[
    B \subset \bigcap_{i=i_0+1}^{k} \left( \tilde A_i \cap \tilde A'_i \right)^c.
  \]
  From the independence and bounds above we deduce
\begin{equation}
    \P(B) \le \left(1-e^{-2c_3\vep^{-2}} \right)^{k-i_0}\le \left(1-e^{-2c_3\vep^{-2}} \right)^{k-9-2|\log_2\vep|}.\label{Bbound}
\end{equation}

  Now, on the event $B^c$, there is some $i\in(i_0,k]$ so that $I_+:=[s,s+2^i]$ and $I_-:=[s-2^i,s]$ are both $\vep$-slow.
  This implies that $R^{\bs{S}}_{I_+}\le 1+\vep 2^{i/2}\le 2\cdot 2^{i/2}\vep$ and $R^{\bs{S}}_{I_-}\le 1+\vep 2^{i/2}\le 2\cdot 2^{i/2}\vep$.
  It follows that $R^{\bs{S}}_{I_+\cup I_-}\le 4\vep 2^{i/2}$.
  Now, there is some dyadic interval $I\subset I_+\cup I_-$ of length $2^i$, with $s \in I\subset [0,2^k]$, and clearly $R^{\bs{S}}_{I}\le 4\vep 2^{i/2}$ as well.
  Thus $I$ is $(4\vep)$-slow for $\bs{S}$.
 
This shows that for any $\vep>0$ etc., the probability that $s$ is not contained in any $4\vep$-slow dyadic interval in $[0,2^k]$ is at most \eqref{Bbound}.  Applying this result to $\vep/4$ instead, yields the desired conclusion for $k>9+2|\log_2(\vep/4)|$.  Since the claim is trivial for $k\le 9+2|\log_2(\vep/4)|$ this completes the proof.
\end{proof}

\medskip

We shall rely on the following calculus exercise.

\begin{lemma}
  \label{L:g_property}
  Fix $a\in (0,1)$, and consider the function $g:\R\to\R$ defined by
  \[ g(\theta) := \theta a-\log(\cosh(\theta)). \]
  Then, the following statements are satisfied.
  \begin{itemize}[nosep]
  \item[(i)] The equation $g(\theta)=0$ has exactly two roots, $0$ and $\theta_a>0$.
  \item[(ii)] For $\theta>\theta_a$ we have that
    $g(\theta)<0$.
  \item[(iii)] The root $\theta_a$ satisfies the bounds
    \[ 2a \le \theta_a \le 2a\left(1+\frac{1}{(1-a)^2}\right). \]
  \end{itemize}
\end{lemma}

\begin{proof}
  {\it Parts $(i)$ and $(ii)$.} 
  Note that $g(0)=0$, $g'(0)=a>0$. Hence, there is a $\theta_1>0$ such that 
$g(\theta_1)>0$. Furthermore, since $a<1$ we also have that 
$\lim_{\theta\to\infty}g(\theta)=-\infty$. Hence there is a 
$\theta_2>\theta_1$
such that $g(\theta_2)<0$. By the intermediate value 
theorem it 
follows that there is a $\theta_a\in (\theta_1,\theta_2)$ (in 
particular $\theta_a>0$) such that 
$$
g(\theta_a)=\theta_a a -f(\theta_a)=0. 
$$
Furthermore, since $g''(\theta)=-(\cosh^2\theta)^{-1}$, it follows 
that $0$ and $\theta_a$ are the only roots of $g(\theta)=0$ and that 
$$
g(\theta)<0\quad{\rm for} \quad \theta>\theta_a. 
$$

\medskip

\noindent {\it Part $(iii)$.}
Note that for $\theta\in\mathbb R$ we have that $\cosh\theta\ge 
\frac{e^{\theta}}{2}$. Hence, since  $\theta_a 
a=\log(\cosh\theta_a)$ we see that

$$
\theta_a-\log 2\le\theta_a a. 
$$
From here we conclude that 

\begin{equation}
  \label{upperboundtheta}
\theta_a\le\frac{\log 2}{1-a}. 
\end{equation}
On the other hand 
we also have 
for $\theta\in\mathbb R$ the 
bounds $1+\frac{\theta^2}{2}\le\cosh(\theta)\le 
e^{\theta^2/2}$. Hence again using that  $\theta_a 
a=\log(\cosh\theta_a)$, we conclude that

$$
\log\left(1+\frac{\theta_a^2}{2}\right)\le\theta_a a\le\frac{\theta^2_a}{2}. 
$$
Now, since we also have for $x>-1$ that $\log(1+x)\ge\frac{x}{1+x}$, 
we get 
$$
\frac{\theta_a^2}{2+\theta_a^2}\le\theta_a a\le\frac{\theta^2_a}{2}. 
$$
Substituting the upper bound (\ref{upperboundtheta}) into the denominator of the left-hand side of the above inequality we conclude that 

$$
\frac{\theta_a}{2+\frac{2}{(1-a)^2}}
\le \frac{\theta_a}{2+\frac{(\log 2)^2}{(1-a)^2}}
\le a\le\frac{\theta_a}{2}. 
$$
Inverting these bounds we finish the proof.
\end{proof}

Our next lemma gives a bound on certain functional of the range of a simple random walk. 

\begin{lemma}
  \label{L:range_sup}
  There is a constant $\gamma > 0$ such that for all $l\ge 0$, 
  \[
    \E\left[\sup_{k\ge 0} \left\{R_k^{\sss{\bs{S}}}\sqrt{l} - \gamma(k+l)\right\}\right]\le 0. 
  \]
\end{lemma}

\begin{proof}
  The proof has 4 steps:
  \begin{enumerate}[nosep]
  \item Estimate the probability that the walk ever exceeds $at+b$.
  \item Estimate $\sup \left\{R^{\sss{\bs{S}}}_k-(ak+b)\right\}$.
  \item Prove the claim for some $\gamma$ and large $l$.
  \item Upgrade to some (other) $\gamma$ and all $l$.
  \end{enumerate}
  
  \paragraph{Step 1.}
  Let $a\in (0,1)$ and $\theta_a$ be the  positive root of $a\theta=\log\cosh(\theta)$ from \cref{L:g_property}.  Define 
   \[ M_n(a) = e^{\theta_a S_n-n \log\cosh(\theta_a)}=e^{\theta_a(S_n-na )+n g(\theta_a)}=e^{\theta_a(S_n-na)}. \]
  Then $(M_n(a))_{n \ge 0}$ is a Martingale. 
  Let $b>0$ and define the (possibly infinite) stopping time $\tau = \tau_{a,b}$ by 
  \[ \tau_{a,b} = \inf\{k\ge 0: S_k\ge ak+b\}. \]
  Up to and (when $\tau$ is finite) including at $\tau$ we have $S_n < an+b+1$.
  It follows that the stopped martingale $M_{n\wedge \tau}(a)$ is bounded by $e^{\theta_a}(b+1)$.
  We can thus apply the optional stopping theorem to $M_n(a)$ at the stopping time $\tau$.
  
  On the event that $\tau$ is finite we have 
  \begin{equation}
    a\tau + b \le S_{\tau}\le a\tau+b+1, \label{Ytau}
  \end{equation}
  so on this event we have
  \[ e^{\theta_a b} \leq M_\tau{(a)} \leq e^{{\theta_a} (b+1)}. \]
  Since $S_n$ grows sub-linearly, we also have that $M_n{(a)}\to 0$ as $n\to\infty$ almost surely.
  The optional stopping theorem now implies that
  \[ 1 = E [M_\tau(a)]
    = 
    E[M_\tau(a)\indic{\tau<\infty}]. \]
  The bounds on $M_\tau(a)$ imply that
  \begin{equation}
    \label{oo}
    e^{-\theta_a (b+1)}\le P(\tau_{a,b}<\infty)\le e^{-\theta_a b}.
  \end{equation}

  \paragraph{Step 2.}
  Let $a\in (0,1)$ and 
  \begin{equation}
    \label{es}
    \bar{S} := \sup_{k\ge 0} \left\{R^{\sss{\bs{S}}}_k-ak\right\}.
  \end{equation}
  Note that the supremum in (\ref{es}) is attained since $a>0$ and the range grows sublinearly (e.g.~by the law of the interated logarithm),
  and that $\bar{S} \geq 0$ (from $k=0$).
  Now, for each $j\in\N$, consider the event
  \[ A_j := \left\{\bar{S}\ge j\right\}. \]
  Since the supremum in the definition of $\bar{S}$ is attained, then for $j\ge 0$, $\bar{S}\ge j$ implies that for some $n$ we have $|S_n| \ge (an+j)/2$.
  Thus
  \[
    \{\tau_{a,j}<\infty\} \subset A_j
    \subset \{\tau_{a/2,j/2}<\infty\} \cup \{\tau'_{a/2,j/2}<\infty\},
  \]
  where $\tau'$ is the same stopping time for $-\bs{S}$.
  It follows that for $j\ge 0$,
  \[ P(\bar{S}\ge j)\le 2 P(\tau_{a/2,{j/2}}<\infty). \]

  We then have from the upper bound in (\ref{oo}) that $\bar{S}$ satisfies
  \begin{align*}
    E\big[\bar{S}\big]
    &= \int_0^\infty P(\bar{S}\ge t)dt \le 2 \int_0^\infty e^{-\theta_{a/2}(t-1)/2}dt= \frac{4}{\theta_{a/2}} e^{(\theta_{a/2})/2}.
  \end{align*}
  Translating by $b$, it follows that
  \begin{equation}
    \label{thetaab}
    E\left[\sup_{k\ge 0}\{R^{\sss{\bs{S}}}_k-(ak+b)\}\right]
    \le -b + \frac{4 e^{(\theta_{a/2})/2}}{\theta_{a/2}} .
  \end{equation}

  \paragraph{Step 3.}  
We claim that there is an $L>0$ such that 
  \begin{equation}
    \label{int2}
 {\sup_{l\ge L}}   E \left[\sup_{k\ge0} \left\{R_k^{\sss{\bs{S}}}\sqrt{l}-{4}(k+l)\right\}\right] \le 0. 
  \end{equation}
  To see this, divide by $\sqrt{l}$, and note that it suffices to prove that {for $l\ge L$}
  \[
    E \left[\sup_{k\ge0} \left\{R_k^{\sss{\bs{S}}}-({4}/\sqrt{l})k+{4}\sqrt{l})\right\}\right] \le 0.
  \]
  This is bounded in (\ref{thetaab}) with $a={4}/\sqrt{l}$ and $b={4}\sqrt{l}$ by
  \[
    - {4}\sqrt{l} + \frac{{4} e^{(\theta_{{2}/\sqrt{l}})/2}}{\theta_{{2}/\sqrt{l}}} .
  \]
  From \cref{L:g_property}(iii), for $l$ large enough we have ${4}/\sqrt{l} \le \theta_{{2}/\sqrt{l}} \le {10}/\sqrt{l}$, so the bound becomes
{  \[-4\sqrt{l}+\frac{4e^{5/\sqrt{l}}}{4/\sqrt{l}}=\sqrt{l}[-4+e^{5/\sqrt{l}}].\]}
This is clearly negative for $l$ large enough, which verifies the claim.


  \paragraph{Step 4.}
  Applying the bound \eqref{thetaab} with $a=4/\sqrt{l}$ and $b=4\sqrt{l}$ shows that for any $L$ there exists a constant $C=C(L)$ such that 
  \begin{equation}
    \nonumber
    \sup_{1\le l\le L}
    \E \left[\sup_{k\ge0} \left\{R_k^{\sss{\bs{S}}}\sqrt{l}-{4}(k+l)\right\}\right]\le C. 
  \end{equation}
  From here we can see that, with {$\gamma:=4+C$}
  \[
    \sup_{1\le l\le L}E\left[\sup_{k\ge0}
      \left\{R_k^{\sss{\bs{S}}}\sqrt{l}-{\gamma}(k+l)\right\}\right]\le 0. 
  \]
  Since the quantity $R_k^{\sss{\bs{S}}}\sqrt{l}-x(k+l)$ is monotone decreasing in $x$, we conclude from  \eqref{int2} that  
\[
    \sup_{l\ge L}E\left[\sup_{k\ge0}
      \left\{R_k^{\sss{\bs{S}}}\sqrt{l}-{\gamma}(k+l)\right\}\right]\le 0,
  \]  
  which completes the proof.
\end{proof}

\subsection{Proof of Theorem \ref{upper-bound-theorem}}

Consider two independent $1$-dimensional simple random walks $\bs{X}$ and $\bs{Y}$ and a timing rule $\bs Q$ that is $\bs Y$-adapted.
Consider the two-dimensional random walk $\bs Z$ given by $\bs{X}$, $\bs{Y}$, and the rule ${\bs Q}$.
Recall that for each $t \in \Z_+$, $N_t$ denotes the number of $\bs{X}$-steps that $\bs Z$ takes by time $t$.
Since there will be multiple time scales (for 1-dimensional and 2-dimensional walks respectively) in what follows, we will distinguish between intervals on these time scales by using the subscript notation 
$[a,b]_{\sss{\bs{X}}}$ or $[a,b]_{\sss{\bs{Z}}}$.
The $\bs{X}$ walk takes $b-a$ steps during the interval $[a,b]_{\sss{\bs{X}}}$ and $N_v-N_u$ steps during the interval $[u,v]_{\sss{\bs{Z}}}$.
The range of a walk $\bs{X}$ over an interval $I = [a,b]_\ssX$ is denoted $\mc{R}^{\bs{X}}_I$, and its cardinality is $R^{\bs{X}}_I$.
  (The trivial bound for this is $R^{\bs{X}}_I \leq |I|+1$.)  Similarly the range of the walk $\bs{Z}$ over the interval $I=[a,b]_{\bs{Z}}$ is denoted by $\mc{R}^{\bs{Z}}_I$ and its cardinality by $R^{\bs{Z}}_I$.

Recall that a 1-dimensional simple random walk $\bs{X}$  is $\vep$-slow on an interval  $J_{\sss{\bs{X}}}$
if $R^{\bs{X}}_{J_{\bs{X}}}-1\le \vep \sqrt{|J_{\bs{X}}|}$.

The main steps in the proof are as follows. 
\begin{enumerate}[nosep]
\item Decompose the interval $[0,N_t]_{\sss{\bs{X}}}$ into $\vep$-slow intervals and a complementary set, and derive from these a certain decomposition of $[0,t]_\ssZ$.
\item Estimate the range of $\bs{Z}$ during the slow intervals.
\item Estimate the range of $\bs{Z}$ in the remaining times.
\item Estimate the size of the complementary set.
\item Combine the bounds, set the parameters and deduce the result.
\end{enumerate}

\paragraph{Step 1.}
For $\vep>0$ and an interval $I_\ssX$ 
define the collections of intervals
\[
  \mathcal D_{\vep, I_{\sss{\bs{X}}}}
  :=\{J_{\sss{\bs{X}}}\subset I_{\sss{\bs{X}}}
  : J_{\sss{\bs{X}}} \text{ is a dyadic $\vep$-slow interval for $\bs{X}$}\}.
\]
Also define the subcollection of ``maximal'' intervals
\[
  \mathcal A_{\vep, I_{\sss{\bs{X}}}}
  := \{
  J_{\sss{\bs{X}}}\in\mathcal D_{\vep, I_{\sss{\bs{X}}}}: \text{ there is no }
  J'_{\sss{\bs{X}}}\in\mathcal D_{\vep, I_{\sss{\bs{X}}}} \text{ such that } J_{\sss{\bs{X}}}\subsetneq J'_{\sss{\bs{X}}}\}.
\]

Any two dyadic intervals are either disjoint (except possibly their endpoints) or one is contained in the other.
It follows that the intervals in $\mathcal A_{\vep, I_{\sss{\bs X}}}$ are step-disjoint (i.e. two intervals in $\mc{A}_{\vep, I_{\sss{\bs{X}}}}$ may share an endpoint, but no more). 
The union of the intervals in $\mc{A}_{\vep, I_{\sss{\bs{X}}}}$ is denoted by 
\[
  A_{\vep, I_{\sss{\bs{X}}}}
  := \bigcup_{J_{\sss{\bs{X}}}\in\mathcal A_{\vep, I_{\sss{\bs{X}}}}} J_{\sss{\bs{X}}}.
\]
Note that whenever $I_{\sss{\bs{X}}}\subset I'_{\sss{\bs{X}}}$, we have that
\begin{equation}
  \label{mon}
  A_{\vep, I_{\sss{\bs{X}}}}\subset A_{\vep, I'_{\sss{\bs{X}}}}.
\end{equation}

In what follows we will work with the collection $\mathcal A_{\vep, [0,N_t]_\ssX}$.
In order to simplify notation, we will write $\mathcal A_{\vep, N_t}$ and $A_{\vep, N_t}$ instead of $\mathcal A_{\vep, [0,N_t]_{\sss{\bs{X}}}}$ and $A_{\vep, [0,N_t]_{\sss{\bs{X}}}}$ respectively.

The walk $\bs{Z}$ interlaces the steps of $\bs{X}$ with those of $\bs{Y}$.
In light of this, there is a natural correspondence between intervals of $\bs{X}$, intervals of $\bs{Y}$, and intervals of $\bs{Z}$.
For an interval $I=[a,b]_\ssZ$ we let $N_I := [N_{a},N_{b}]_\ssX$, and let $M_I=[M_a,M_b]_\ssY$.
The converse mapping is a bit more delicate, since there is some flexibility in how to handle the endpoints of the intervals.

Formally, for  each interval $J\in[0,\infty)_\ssX$ 
we associate an interval $I \in [0,\infty)_\ssZ$ corresponding to the time
passed in the timeline of ${\bs Z}$.
Suppose that $J_\ssX=[a,b]_\ssX$, for some  $a\le b \le 2^j$.
Let $U_-=\inf\{u\ge 0: N_u=a\}$ and $U_+=\sup\{u\ge 0: N_u=b\}$.
We now define the map $\phi$ by $\phi(J_\ssX)=[U_-,U_+]_\ssZ$.
Note that $\phi(J_\ssX)$ is a random interval.
 
We are in interested in particular in the images of slow dyadic intervals.
For each $j\ge 0$, let
\[\cB_{j}=\big\{ \phi(J) : J \in \mathcal{A}_{\vep,[0,2^j]_\ssX} \big\} \qquad \text{ and } \qquad
  B_{j} = \bigcup_{I\in\cB_{j}} I.
\]
The decomposition of the timeline of $\bs Z$ is into the intervals of $\cB_j$ and the complement of those.

\paragraph{Step 2.}

We will prove that for the constant $\gamma$ from \cref{L:range_sup},
for every $j\ge 0$ and  $I=[a,b]_\ssX$ we have 
\begin{equation}
  \label{ss}
  \E \Big[ \big(R^{\bs{Z}}_{\phi(I)} - \gamma \vep |\phi(I)| \big) \indic{I\in\cA_{\eps,[0,2^j]_\ssX}} \Big] \le 0.
\end{equation}
Roughly speaking this says that for a slow interval for $\bs{X}$, the expected range of $\bs{Z}$ in the corresponding time interval is comparably small.  Note that if $b>2^j$ the indicator is 0 so \eqref{ss} is trivial.
Otherwise, let $\phi(I) = [U_-,U_+]_\ssZ$, so that $a=N_{U_-}$ and $b=N_{U_+}$.
For $u_-\le u_+$ define $I'_{u_-,u_+}=
\left[M_{u_-},M_{u_+}\right]_\ssY$. 
Note that $I'_{U_-,U_+}$ is a
random interval.
To simplify notation we will in what follows sometimes drop the
subindices
and write $I'$ instead of $I'_{U_-,U_+}$.
At any time $i\in[U_-,U_+]_\ssZ$ the $X$ coordinate of $Z_i$ is some value of $\bs{X}$ taken during $I$, and the $Y$ coordinate is a value of $\bs{Y}$ during $I'$.
Consequently, $\bs{Z}$ is constrained to a rectangle and we have the bound
\[
R^{\bs{Z}}_{\phi(I)}\le R^{\bs{X}}_{I} \cdot R^{\bs{Y}}_{I'}.
\]
Suppose that 
$I$ 
is $\vep$-slow  for $\bs{X}$.  
Then 
\[
R^{\bs{Z}}_{\phi(I)}
  \leq \eps \sqrt{|I|} \cdot R^{\bs{Y}}_{I'}.
\]
Note that $|\phi(I)| = |I| + |I'|$.
Then  with $\gamma$ from \cref{L:range_sup} we have 
\begin{align*}
 R^{\bs{Z}}_{\phi(I)} - \gamma \vep |\phi(I)|
  &\leq \eps( \sqrt{|I|} \cdot R^{\bs{Y}}_{I'}- \gamma (|I|+|I'|)).
\end{align*}
Taking expectations we have that 
\begin{align*}
&\E\big[ (R^{\bs{Z}}_{\phi(I)} - \gamma \vep |\phi(I)|)\indic{I
                 \in \mc{A}_{\vep,[0,2^j]_{\ssX}}}\big]\\&  \le \sum_{u_-=0}^\infty \vep \E \Big[\Big(\sqrt{|I|} \cdot R^{\bs{Y}}_{I'_{U_-,U_+}}- \gamma (|I|+|I'_{U_-,U_+}|) \Big)\indic{I \in \mc{A}_{\vep,[0,2^j]_\ssX}}\indic{U_-=u_-}\Big]\\
&  \le \sum_{ u_-=0}^\infty\vep \E \Big[\E \big[\sqrt{|I|} \cdot R^{\bs{Y}}_{I'_{u_-,U_+}}- \gamma (|I|+|I'_{u_-,U_+}|) \big|\mc{G}_{u_-} \big]\indic{I \in \mc{A}_{\vep,[0,2^j]_\ssX}}\indic{U_-=u_-}\Big],
\end{align*}
where we have used the fact that both indicator functions are $\mc{G}_{u_-}$-measurable.  
Since the timing rule is $\bs{Y}$-adapted, conditioned on $\mc{G}_{u_-}$ the subsequent steps of $\bs{Y}$ are still an independent simple random walk.  Thus, using \cref{L:range_sup} we have a.s.,
\[\E \Big[\sqrt{|I|} \cdot R^{\bs{Y}}_{I'_{u_-,U_+}}- \gamma (|I|+|I'_{u_-,U_+}|) \Big|\mc{G}_{u_-} \Big]
  \leq 0.
\]
This completes the proof of \eqref{ss}.

\paragraph{Step 3.}

Given a subset $A\subset [0,\infty)_{\bs{Z}}$, denote by
\[
  \mathcal R^{\bs{Z}}_A=\{x\in\mathbb Z^2: Z_n=x, \text{ for some }\ n\in
  A\}
\]
the range of the random walk at times in $A$, and its cardinality by $R^{\bs{Z}}_A$ (thus generalizing this notation from intervals to arbitrary sets).
If we sum \eqref{ss} over all dyadic intervals in $[0,2^j]_\ssX$ (with contributions only from intervals in $\cA_{\vep,[0,2^j]_\ssX}$), we get the following:
\begin{align}
  \E[R^{\bs{Z}}_{B_{j}}]\le 
  \E\Big[ \sum_{I\in \cA_{\vep,[0,2^j]_\ssX}} R^{\bs{Z}}_{\phi(I)}\Big]
  &\leq  \gamma\varepsilon\E\Big[\sum_{I\in \cA_{\vep,[0,2^j]_\ssX}}
  |\phi(I)|\Big]\\
  &
  \le \gamma\vep\sum_{i=0}^{2^j}\E\Big[\sigma_{i+1}-\sigma_i\Big]\le 2K \gamma \vep 2^j.
  \label{eq:slow_range}
\end{align}
Here, the third inequality holds since the intervals of $\cA_{\vep,[0,2^j]_\ssX}$ are step-disjoint within $[0,2^j]_\ssX$,
while the last inequality uses condition (\ref{eee}).

\paragraph{Step 4.}  
To bound the range of $\bs Z$ in the complement of the slow intervals, first note that
\begin{align}
\E \left[\#([0,\sigma_{2^j}]_\ssZ\setminus B_j)\right]&\le \sum_{i=0}^{2^j}\E\left[\sum_{s=\sigma_i}^{\sigma_{i+1}-1}\indic{s \notin B_j}\right].
\end{align}
Now $s \in B_j$ if and only if $s \in \phi(J)$ for some $J \in \mc{A}_{\vep,[0,2^j]_{\ssX}}$, and  from the definition of $\phi$ we have for $s\in [\sigma_i,\sigma_{i+1})$ that $s\in \phi(J)$ if and only if $\sigma_i \in \phi(J)$.  It follows that
\begin{align}
  \nonumber
  \E \left[\#([0,\sigma_{2^j}]_\ssZ\setminus B_j)\right]
  &\le \sum_{i=0}^{2^j}\E\left[ \indic{i\notin A_{\varepsilon,[0,2^j]_\ssX}} (\sigma_{i+1}-\sigma_i)\right]\\
  \label{step42}
  &= \sum_{i=0}^{2^j}\E\left[\indic{i\notin A_{\varepsilon,[0,2^j]_\ssX}}    \E\left[\sigma_{i+1}-\sigma_i|\mathcal G_{\sigma_i}\right]\right]\\
  & \leq K 2^{j+1} \left( 1 - e^{-\delta/\eps^2} \right)^{j-9-2|\log_2(\eps/4)|}.
\end{align}
In the second inequality we used the fact that the indicator 
is $\mathcal G_{\sigma_i}$-measurable (slow intervals are determined purely by the trajectory of $\bs X$), and in the last inequality we used \cref{P:slow_diad} and Condition (\ref{eee}).

\paragraph{Step 5.}
Combining the bounds of steps 3 and 4 we get that
\begin{align*}
  \E\left[
    R^{\bs{Z}}_{[0,\sigma_{2^j}]_{\bs{Z}}}
  \right]& \le
  \E\left[R^{\bs{Z}}_{B_j}\right]+\E \left[\#([0,\sigma_{2^j}]_\ssZ\setminus
    B_j)\right]\\
&  \le 2K\gamma\varepsilon 2^j+
    K 2^{j+1} \left( 1 - e^{-\delta/\eps^2} \right)^{j-9-2|\log_2(\eps/4)|}.
\end{align*}
Now, given $t\ge 0$, choose $j\ge 0$ such that
\begin{equation}
  \label{jt}
  2^{j-1} \le t < 2^j.
\end{equation}
By definition, the number of steps made by the ${\bs{X}}$ coordinate up to time $\sigma_{2^j}$ in the timeline of $\bs Z$ is $2^j$.
Thus we have that $t\le 2^j\le\sigma_{2^j}$.
Therefore,
\begin{align}
\nonumber
\E\left[R^{\bs{Z}}_{[0,t]_{\bs{Z}}}\right]&\le
\E\left[R^{\bs{Z}}_{[0,\sigma_{2^j}]_{\bs{Z}}} \right]\\
\label{s61}
&\le K 2^{j+1} \left(\gamma\varepsilon +
    \left(1 - e^{-\delta/\eps^2} \right)^{j-9-2|\log_2(\eps/4)|} \right).
\end{align}

Since $\gamma,\delta$, and $K$ are constants, all that remains is to find $\eps$ so as to minimize this bound.
Set $\eps = \sqrt{2\delta/\log j}$, so that $(1 - e^{-\delta/\eps^2}) = 1-1/\sqrt{j} \leq e^{-1/\sqrt{j}}$.
Thus $\eps \sim c/\sqrt{\log\log t}$ as $t\to\infty$.  
Since from (\ref{jt}) we see that $j\to\infty$ as $t\to\infty$, we see that for $t$ large enough, $j-9-2|\log_2(\eps/4)| \geq j/2$.
Thus so we find for $t$ sufficiently large
\[
  \E\left[R^{\bs{Z}}_{[0,t]_{\bs{Z}}}\right]
  \leq K 2^{j+1} \left( \gamma \vep + e^{-\sqrt{j}/2} \right)
  \leq K 2^{j+1} \left( \gamma+1 \right) \vep.
\]
Now, from (\ref{jt}) we see that $2^j\le 2t$ and $j\ge \frac{1}{\log
  2}\log t$, so that
\[
  \E\left[R^{\bs{Z}}_{[0,t]_{\bs{Z}}}\right]
  \le
  4K(\gamma+1)\vep t
  \le
  C\frac{t}{\sqrt{\log\log t}},
\]
for some constant $C>0$, for all sufficiently large $t$.

\section{Lower bound on the range}
\label{sec:LB}

The main idea is to show that the initial visits to a level (a horizontal line) do not take too long.
If little time is spent in each level, then the process has visited many levels, hence made many vertical steps.
The number of vertical steps is equal to the range, which is therefore large.

Thus we start by focusing on a single level $\Z\times\{y\}$.
The process enters the level at some sequence of locations $ (a_i,y)$, which we will simply refer to as $a_i$.
Each time it enters the level it performs a simple random walk within the level until it reaches a previously unvisited vertex, at which time it exits to level $y\pm 1$.
The location of the next entry to the level is determined by the exit point and random events that occur while the process is outside the level.
To avoid dealing with those dependencies, we identify events that hold with high probability uniformly in this ``external'' randomness.

For this section we use the following construction of the balanced excited random walk via stacks of instructions (one stack per vertex): the first instruction at each vertex being a vertical step, and all subsequent ones are horizontal steps.

Fix $y \in \Z$.  We denote by $\F_y$ the $\sigma$-algebra generated by the stacks of all vertices in level $y$. 
In particular, the durations and exit nodes of the first $n$ visits to a level are determined by the entry points $a_1,\dots,a_n$ together with $\F_y$.
Let $\F^v_{y}$ denote the $\sigma$-algebra generated by the vertical instructions at levels $y\pm 1$, and $\F'_y$ the smallest $\sigma$-algebra containing $\F_y$ and $\F^v_y$.
(Note that -- unlike the $\F_y$ -- these $\sigma$-algebras are not independent of each other.)

In a similar manner to Section \ref{sec:recurrence} we will consider families of random walks in level $y$.
In the first instance these walks will be measurable with respect to $\F_y$ as detailed next.

\subsubsection*{A collection of random walks in level $y$.}

Given a finite sequence $\bs{a}=(a_1,\dots,a_n)$ of points in $\Z$ and sequences of instructions $\bs{I}_y=(I((x,y),k)_{x \in \Z, k \in \N}$, the $\bs{a}$-family of walks is defined as follows:
We first start a walk at position $(a_1,y)$.
It uses the first instruction at $(a_1,y)$ (which is a vertical step) and then stops, using up exactly $\LL_1=1$ instructions.
Inductively, the $i$th walk starts at $(a_i,y)$, and uses the instructions following the ones used by previous walks, 
always using the first unused (by itself or any of the previous walks) instruction at its current location.
Once the walk visits a previously unvisited site, it uses the first instruction at that site (which is a vertical instruction) and stops.
The \textbf{lifetime} $\LL_i$ of the $i$-th walker is the number of instructions used by the walk, which is 1 more than the number of horizontal steps made by this walk.
The instruction stacks $\bs{I}$ will be i.i.d.~over sites (and balanced) so that all of the walks eventually stop almost surely, and the above process is well defined.

\begin{lemma}
  Let $\bs{a}=(a_1,\dots,a_n)$ be some sequence of starting points in level $y$, and $\bs{a}'=(a'_1,\dots,a'_n)$ be any permutation of $\bs{a}$.
  Then the total lifetime of the $\bs{a}$-family of walks is equal to the total lifetime of the $\bs{a}'$-family of walks,
  i.e.~$\sum_{i=1}^n \LL_i = \sum_{i=1}^n \LL'_i$.
\end{lemma}

\begin{proof}
  By the Abelian property (Proposition \ref{prop:abelian1}) the exact same set of instructions at each vertex will be used in the two scenarios.
  The total number of instructions used is equal to the total lifetime.
\end{proof}

For a sequence $\bs{a}$ of starting points, we henceforth denote by $\cL_{y,\bs{a}} = \cL_{\bs{a}}$ the number of instructions used by the walks starting at points in $\bs{a}$.
For now we shall omit the $y$ subscript, while we consider only a single level, but will use it later when working with multiple levels.
Thus for any sequence $\bs{a}$, we have that $\cL_{\bs{a}}$ is a random variable measurable in $\F_y$.
In light of the last lemma, the order of elements in $\bs{a}$ is immaterial, and we may think of $\bs{a}$ as a multi-set rather than a sequence.
Moreover, we henceforth assume that the $a_i$ are non-decreasing, and will study the durations of the visits under that assumption.
For the $k$'th walk started at $a_k$, let $[u_k+1,v_k-1]$ be the maximal contiguous interval - of vertices visited by the first $k-1$ walks - that contains $a_k$.  If $a_k$ has not previously been visited, we write $u_k=v_k=a_k$, so that this is an empty set.
(In that case we have $\LL_k=1$.)

\begin{remark}
  If $a_i\equiv a$ for all $i$ (a constant sequence), then the process in the level is equivalent to one-dimensional IDLA (see e.g.~\cite{IDLA1} where this has been studied also in higher dimensions), and one has $\sum_{k\le n} \LL_k \asymp n^3$. 
  However, in the BERW model, each vertex $(x,y)\ne (0,0)$ is used as an entry point to the level at most twice, depending on the unique vertical steps from $(x,y\pm1)$.
  Our proof exploits this fact.
\end{remark}

\subsubsection*{Feasible sets and the surplus}
For $(x,y)\ne (0,0)$ we denote by $U^y_x \in\{0,1,2\}$ the number of vertical instructions pointing at $(x,y)$.
The starting point $(0,0)$ is slightly different:
We define $U_{0}^0$ to be 1 plus the number of vertical instructions pointing at $(0,0)$.
For fixed $y$, the $(U^y_x)_{x\in\Z}$ are independent random variables that are  $\sim \Bin(2,1/2)$ for $(x,y)\ne (0,0)$ (and $\sim 1+\Bin(2,1/2)$ for $(0,0)$), measurable w.r.t.~$\F^v_{y}$, and independent of $\F_y$.

We say that a sequence $\bs{a}$ of starting points is \textbf{feasible} (for level $y$) if no $x$ appears in $\bs{a}$ more than  $U^y_x$ times.  
For an interval $I\subset\Z$ we define its \textbf{surplus} $S^y_I$ by
\[
  S^y_I := \sum_{x\in I} (U^y_x-1).
\]
Note that the surplus of an interval of length (number of vertices) $m$ has distribution $\Bin(2m,1/2)-m$.
We denote by $M^y_n$ the maximal surplus of an interval of length at most $n$ in $[-n^2,n^2]$ in level $y$:
\[
  M^y_{n} := \max \{S^y_{[a,b)} : -n^2\le a\le b < n^2, b-a\leq n\}.
\]
We shall consider the event
\[
  E=E_{n,y} := \big\{M_{n}^{y} \leq \sqrt{6n\log n}\big\}.
\]

\begin{lemma}
  \label{lem:badexcess}
  For each $n\ge 10$ and $y\in \Z$, $\P(E_{n,y}^c) \leq 3 n^{-3}$.
\end{lemma}

\begin{proof}
  For an interval $I$ of length $m$, and $S^y_I \sim \Bin(2m,1/2)-m$ (unless $y=0$ and $0\in I$).  The latter can be written as $\frac12 \sum_{i\le 2m} X_i$,
  where the $X_i$ are i.i.d. $\pm1$ random variables with mean 0.
  By the standard Chernoff bound for this sum, we have 
  \[
    \P(S^y_I \geq \ell) = \P\Big(\sum_{i\le 2m} X_i \geq 2\ell\Big)
    \leq e^{-\ell^2/m}.
  \]
  For an interval $I$ of length \emph{at most} $m\le n$ we thus have
  \[
    \P(S^y_I \geq \ell) \leq e^{-{\ell^2/|I|}} \leq e^{-\ell^2/n}.
  \]
  There are at most $2n^3$ intervals $I$ to consider in the definition of $M^y_n$, and the claim follows by a union bound:
  \[
    \P(E^c) \leq 2 n^3 e^{-(6n\log n)/n} = 2n^{-3}.
  \]
  
  The case where $y=0$ is minutely different since intervals containing $0$ have $S^y_I\sim 1+Bin(2m,1/2)-m$.
  There are less than $n^2$ such intervals, and for $\ell\ge 1$  each has
  \[
    \P(S_I^0 \ge \ell) \leq e^{-(\ell-1)^2/n}.
  \]
  For $n\ge 10$ we have $\sqrt{6n\log n}-1\geq \sqrt{5n\log n}$, so the contribution from these intervals to $\P(E^c)$ is at most $n^{-3}$.
\end{proof}

The main significance of the maximal surplus comes from the following obervation.

\begin{lemma}\label{L:bdry_near}
  Let $\bs{a}$ be a feasible \emph{increasing} sequence of starting points in $[-n^2,n^2]$.
  Then for each $k=1,\dots n$, when the $k$-th walker starts (at $a_k$), the distance to the nearest unvisited point is at most $M^y_n$.
\end{lemma}

\begin{proof}
  Let $\beta_k$ be the rightmost point visited by the first $k-1$ walks, and $i_k$ be the minimal $i$ such that all vertices in $[a_i,a_k]$ have been visited by the first $k-1$ walks.
  If $a_k$ has not been visited by the first $k-1$ walks then $\beta_k<a_k$ and we set $i_k=k$ (in that case the distance in question is 0 and the claim holds trivially).

  Since $\bs{a}$ is increasing, we know that the first $k-1$ walkers have had $k-i_k$ starting points in the interval $I=[a_{i_k},a_k]$, and the length of the contiguous interval of visited points containing $a_k$ is $k-i_k$.
  Now at most $k-i_k-|I|\le S^y_I-1$ of these walkers can exit to the right of $a_k$, so $\beta_k\le a_k+S_I^y-1$.
  It follows that $\beta_k+1-a_k \leq S^y_I \leq M^y_n$.
\end{proof}

For fixed $r,m\in \N$, let $X$ be a random variable with law given by the exit time from $(0,r)$ of a simple random walk started at $m$, and let $\mathbb{L}_{r,m}=\sum_{i=1}^r X_i$, where $X_i$ are i.i.d.~with law $X$.
We are approaching a central step in the argument for which we will make use of the following Lemmas (which are proved later).

\begin{lemma}\label{L:total_local}
  There exist $c,C>0$ such that for all $r,m\in \N$,
  \[\P(\mathbb{L}_{r,m}>8r^2m) \le C r e^{-cm}.\]
\end{lemma}

Similarly, let $\mathbb{L}'_r$ be the total exit time from $(0,r)$ of $r$ random walks with starting points $\bs{a}_r$ in $(0,r)$ under a probability measure $\P_{\bs{a_r}}$.

\begin{lemma}\label{L:total_local2}
  There exist $c,C>0$ such that for any $1\le r\le n$ and for any given set of $r$ starting points $\bs{a}$ in $(0,r)$
  we have $\P_{\bs{a}_r}(\mathbb{L}'_{r}>4nr^2) \le C r e^{-cn}$.
\end{lemma}

Continue to fix the vertical level $y$.
For an interval $I$, let $\bs{A}_{I}=\bs{A}^y_{I}$ denote the (feasible and $\F^v_y$-measurable) multi-set containing each $x\in I$ exactly $U^y_x$ times, and nothing else.
Let $D=D_{n,y}$ be the event that there is some interval $I=[a,b)\subset[-n^2,n^2)$ with $|I|\leq n$ such that $\cL_{\bs{A}_I} \geq 2^7|I|n^{3/2} \sqrt{\log n}$.
Note that $D\in \F'_y$.

\begin{lemma}\label{L:interval_bd}
  For all $n$ sufficiently large we have $\P(D) \leq 4n^{-3}$.
\end{lemma}

\begin{proof}
  We have
  \[ \P(D) \leq \P(E_{n,y}^c) + \P(E_{n,y}\cap D). \]
  By a union bound and the fact that $\P(E_{n,y}^c) \leq 3n^{-3}$  it suffices to prove that 
  \[\sum_{I}\P\big(E_{n,y},\cL_{\bs{A}_I} \geq 2^7|I|n^{3/2} \sqrt{\log n}\big) \leq n^{-3},
  \]
  where the sum over $I$ is over intervals of length at most $n$ contained in $[-n^2,n^2)$.
  There are at most $2n^3$ such intervals, so it suffices to prove that for each relevant interval $I$,
  \[
    \P(E_{n,y},\cL_{\bs{A}_I} \geq 2^7|I|n^{3/2} \sqrt{\log n}\big) \leq n^{-6}/2.
  \]

  Fix such an interval $I=[a,b)$, and consider the multi-set of entries $\bs{A}_I$, in increasing order.
  Denote its size by $r=|\bs{A}_I| \leq 2|I|$.
  On the event $E_{n,y}$ we have that $S^y_I\leq m := \lceil\sqrt{6n \log n}\rceil$.
  Consider now the $\bs{A}_I$ family of walks.
  By \cref{L:bdry_near} and its proof, when the $k$-th walk starts at $A_k=A_{I,k}$, the previously visited interval containing $A_{k}$ has length at most $k-1<r$, and the distance from $A_k$ to an unvisited vertex is at most $m$.
  If $r\ge 2m$, then
  conditioned on everything that happened during all previous visits to the level, the lifetime $\LL_k$ of the $k$-th walk is stochastically dominated by $X$ (the exit time from $(0,r)$, starting at $m$).
  Since this domination holds uniformly conditional on all previous visits, it also follows that we have the stochastic domination $\sum_{k} \LL_k \prec \sum_{k} X_k$, where there is one term in the sum for each element of $\bs{A}_I$.

  By \cref{L:total_local} with these $r$ and $m$, we find that
  \[ \P(E_{n,y},\cL_{\bs{A}_I} \geq 8r^2m) \leq Cre^{-cm}. \]
  Combining the above with $r\leq 2|I|\leq 2n$ and $m=\lceil\sqrt{6n\log n}\rceil \leq 4\sqrt{n\log n}$, we get (since $8r^2\leq 2^5|I|n$)
  \[ \P\big(E_{n,y},\cL_{\bs{A}_I} \geq 2^7|I|n^{3/2}\sqrt{\log n}\big) \leq 2Cne^{-cm}. \]
  This is indeed smaller than $n^{-6}/2$ for large enough $n$, as needed.

  If $I$ is such that $r < 2m$, then we use \cref{L:total_local2} instead.
  In this case, the total duration of the walks is bounded by the total time for $r$ walks to exit an interval of length $r$.
  Thus for an interval $I$ with $r < 2m$ we have
  \[
    \P(\cL_{\bs{A}_I} \geq 4nr^2) \leq Cre^{-cn} \leq n^{-6}/2
  \]
  for large enough $n$.
  Now, using $r\leq 2m$ and $r\leq2|I|$ we get $4nr^2 \leq 2^4|I|nm \leq 2^7|I|n^{3/2}\sqrt{\log n}$, so the claim follows.
\end{proof}

Towards the proof of \cref{L:total_local} we shall use the following large deviation result for certain random variables.

\begin{lemma}\label{L:basic_ld}
  There exist absolute constants $c,C$ such that for any $n\in \N,p>0$, the Binomial and Negative-Binomial random variables satisfy, respectively,
  \[ \P(\Bin(n,p) > 2np) \leq C e^{-cnp}, \]
  and 
  \[ \P(\Neg(n,p) > 2np^{-1}) \leq C e^{-cn}. \]   
\end{lemma}

\begin{proof}
  For the binomial, observe that if $p\ge 1/2$ the probability in question is 0.  Thus we may assume that $p<1/2$.  Now note that if  $Z=\Bin(n,p)$ and $k\geq 2np$ then $\P(Z=k+1) \leq \P(Z=k)/2$.
  Thus $\P(Z>2np) \leq 2\P(Z=\ouch)$.
  This probability can be written explicitly and estimated:  If $\ouch=n$ then the claim is again immediate.  Otherwise, by Stirling's formula we have, for some universal constant $C$ (below the constant $C$ changes from line to line) 
  \begin{align}
  \P(Z=\ouch)&\le \frac{Cn^{n+1/2}p^{\ouch}(1-p)^{n-\ouch}}{\ouch^{\ouch+1/2} (n-\ouch)^{n-\ouch+1/2}}\\
  &\le \frac{Cn^{n}p^{\ouch}(1-p)^{n-\ouch}}{\ouch^{\ouch} (n-\ouch)^{n-\ouch}}\\
  &\le C 2^{-\ouch}\frac{n^{n}(1-p)^{n-\ouch}}{n^{\ouch} (n-\ouch)^{n-\ouch}}\\
  &=C 2^{-\ouch}\Big( \frac{n(1-p)}{n-\ouch}\Big)^{n-\ouch}\\
  &=C 2^{-\ouch}\Big(1+ \frac{\ouch-np}{n-\ouch}\Big)^{n-\ouch}\\
  &\le C 2^{-\ouch}e^{\ouch-np}\le C (e/4)^{np},
  \end{align}
  and the claim follows.

  For the Negative Binomial, one can use Chernoff's bound for sums of Geometric random variables to obtain the result for $p<9/10$.  For $p\ge 9/10$ we can bound the probability in question by the probability of seeing at least $n$ failures in the first $2n$ trials, which is a Binomial$(2n,1-p)$ probability.
  Again, by including a factor of 2 we can just calculate the probability that this Binomial is exactly equal to $n$, and by a simple application of Stirling, this is exponentially small, uniformly in $p\ge 9/10$.
\end{proof}

\begin{proof}[Proof of Lemma \ref{L:total_local}]
  Note that if $m\ge r$ the claim holds trivially, so assume that $m<r$.
  We shall write the total duration of the $r$ walks as the sum of their cumulative local times at each $s\in(0,r)$.
  Let $N_s$ be the number of walks that reach some $s\in(0,r)$ before exiting the interval.
  Then $N_s$ is $\Bin(r,m/s)$ for $s\in[m,r)$.

  For $s\in [m,r)$,  by \cref{L:basic_ld}, this binomial is exponentially unlikely to be more than twice its expectation, i.e.
  \[
    \P(N_s > 2rm/s) \leq C e^{-crm/s} \leq C e^{-cm}.
  \]
  A walk that reaches $s$ will return to $s$ again before hitting $0$ with probability at most $1-1/2s$ (ignoring that it might reach $r$ and stop there).
  If $N_s\leq k := 2rm/s $ then the total local time at $s$ is bounded by a sum of $\lfloor k\rfloor$ random variables, each being $\Geom(1/2s)$.
  By \cref{L:basic_ld}, we find that $L_s$, the total local time at $s$ satisfies
  \begin{align*}
    \P(L_s > 8rm )
    &\le \P(N_s > 2rm/s) + \P(L_s > 8rm | N_s \leq 2rm/s)\\
    &\le  Ce^{-cm} + Ce^{-crm/s} \leq 2Ce^{-cm}.
  \end{align*}
  
  For $s\leq 2m$ we have $N_s\leq r$, and a sum of the $r$ geometric variables is similarly bounded with high probability by $4rs$, giving the overall bound
  \[ \P(L_s > 8rm) \leq \P(L_s > 4rs) \leq Ce^{-cr} \leq Ce^{-cm}. \]

  It follows that the probability that $L_s>8rm$ for some $s\in (0,r)$ is at most $2Cr e^{-cm}$.
  Thus the probability that $L_s\leq 8rm$ for every $s$ is at least $1 - 2Cr e^{-cm}$.
  On this event $\mathbb{L}_{r,m}\le 8r^2m$.
\end{proof}

\begin{proof}[Proof of Lemma \ref{L:total_local2}]
  We again write the total duration of the $r$ walks as the sum of their cumulative local times at each $s\in(0,r)$.
  If a walk visits $s\in(0,r)$, the probability it exits before returning is at least $2/r$ (uniformly in $s$, with equality only for $s=r/2$).
  Thus the time for each walk at $s$ is bounded by $\Geom(2/r)$, and so the total time at $s$ is bounded by $\Neg(r,2/r)$, which we bound by $\Neg(n,2/r)$.
  By \cref{L:basic_ld}, the probability this exceeds $4nr$ is at most $Ce^{-cn}$. 
  If this does not happen for any $s$, then $\mathbb{L}'_r \leq 4nr^2$.
  Thus $\P_{\bs{a}_r}(\mathbb{L}'_r > 4nr^2) \leq Cre^{-cn}$ as claimed.
\end{proof}

The next step is to enhance \cref{L:interval_bd} to arbitrary (feasible) sets $\bs{A}$ (i.e.~not just $\bs{A}_I$ for intervals $I$).
Recall the event $D=D_{n,y}$ defined above \cref{L:interval_bd}, (that some ``bad'' interval exists), and let $D'=D'_{n,y}$ be the event that there is some feasible set $\bs{A}\subset[-n^2,n^2]$ with size $|\bs{A}|\leq n$ such that $\cL_{\bs{A}} \geq 2^7|\bs{A}|n^{3/2} \sqrt{\log n}$.

\begin{lemma}\label{L:general_bd}
  With these notations, $D'\subset D$, and hence $\P(D') \leq 4n^{-3}$ for all sufficiently large $n$.
\end{lemma}

\begin{proof}
  For a feasible set $\bs{A}$, let $I_1,\dots,I_k$ be the (maximal contiguous) intervals of vertices visited in the level after starting walks at the points of $\bs{A}$.  These depend on the set $\bs{A}$ and on the environment $\F_y$ (but by the Abelian property these intervals do not depend on the order in which the walks are run).
  
  We claim that (deterministically) $\cL_{\bs{A}} \leq \sum_{j=1}^k \cL_{\bs{A}_{I_j}}$.
  Given this claim, on the event $D^c$ we have for every $I_j$ that $\cL_{\bs{A}_{I_j}} \leq 2^7|I_j|n^{3/2}\sqrt{\log n}$.
  Since we have $\sum_{j=1}^k |I_j| = |\bs{A}|$, it follows that on the event $D^c$,
  $\cL_{\bs{A}} \leq 2^7|\bs{A}|n^{3/2}\sqrt{\log n}$.
  This shows that $D^c\subset D'^c$, so it remains to justify the first claim above.
  
  This claim is a natural consequence of the Abelian property when viewed in the right way.
  To see this, let $\bs{B}_j = \bs{A}\cap I_j$, where the intersection is viewed as multi-sets, so $\bs{B}_j$ will include an element twice if $\bs{A}$ does.  Note that $\bs{A}=\cup \bs{B}_j$.
  By the Abelian property, if we start walks in all points of $\bs{A}$, the walks started in $\bs{B}_j$ remain in $I_j$, and so do not interact in any way with walks started in $\bs{A}\setminus \bs{B}_j$.
  It follows that $\cL_{\bs{A}} = \sum_{j=1}^k \cL_{\bs{B}_j}$.
  Since $\bs{B}_j\subset \bs{A}_{I_j}$, we have that $\cL_{\bs{B}_j} \leq \cL_{\bs{A}_{I_j}}$.
  This completes the proof.
\end{proof}

\subsection{Proof of \cref{thm:range}(ii)}

To complete the proof, we return to the representation from \cref{sec:upper} of the BERW as an interlacement of a horizontal random walk and a vertical one.
Recall from there that $\bs{Y}=(Y_s)_{s \ge 0}$ is a simple one dimensional random walk that gives the sequence of vertical steps taken by the process.
This representation is not compatible with the representation in terms of stacks of instructions at each vertex, but the almost sure results that each one implies still hold.
In particular, we rely on two almost sure statements about $\bs{Y}$:
The first is the law of iterated logarithm, stating that almost surely $\limsup Y_n/\sqrt{2n\log\log n} = 1$.
The second is a corresponding result of Kesten \cite{Kesten_local} concerning the maximum of the local time of a simple random walk.
Let $L_{n,y}$ be the number of visits to $y$ in the first $n$ steps of $\bs{Y}$.
Then almost surely
\[
  \limsup_{n\to\infty} \frac{\max_y L_{n,y}}{\sqrt{2n\log\log n}} = 1.
\]

\begin{lemma}\label{L:LIL}
  Almost surely, for all large enough $t$, 
  \begin{itemize}
  \item no level $y$ with $|y| \geq 2\sqrt{R_t\log\log R_t}$ has been visited, and 
  \item no level has been visited (i.e.~entered vertically) more than $2\sqrt{R_t\log\log R_t}$ times.
  \end{itemize}
\end{lemma}

\begin{proof}
  The first follows since the $Y$ coordinate has performed $R_t-1\le R_t$ steps by time $t$, and $R_t \to \infty$ a.s., and using the law of iterated logarithm for $\bs{Y}$.
  (Indeed, it even holds with $\sqrt{2+\eps}$ in place of the $2$.)
  The second claim holds similarly, using Kesten's bound on the maximal local time applied to the walk $\bs{Y}$.
\end{proof}
               
\begin{proof}[Proof of \cref{thm:range}(ii)]
  We may restrict to the almost sure events of \cref{L:LIL}, so that the vertical movements satisfy laws of iterated logarithms.
  Consider the event $G_n = \bigcup_{|y|\leq n} D'_{n,y}$.
  By a union bound and \cref{L:general_bd} we have that $\P(G_n)\le (2n+1)4n^{-3}$ for all $n$ sufficiently large.
  By Borel Cantelli it follows that almost surely $G_n$ occur only for finitely many $n$, i.e. the event $D'_{n,y}$ does not occur for $n$ large enough for any $y$ with $|y|\leq n$.
  Restricted to these almost sure events, we claim that for all $t$ large enough, $R_t \geq C t^{4/7}(\log t)^{-2/7}$.

  Let $\bs{A}_y(t)$ be the multi-set of entrance points to level $y$ up to time $t$ by the BERW,
  \footnote{$(x,y)\in \bs{A}_y(t)$ means that $(x,y)$ was visited at some time $s\le t$ and at time $s-1$ the BERW was not in level $y$ (in particular $(0,0)\in \bs{A}_0(0)$); the multiplicity of any such point other than $(0,0)$ is at most 2}
  and note that $\sum_y |\bs{A}_y(t)|=R_t$.
  Then for each level $y$ (except for the current level at time $t$, which is still ongoing) the total time spent by the process in level $y$ is $\cL_{y,\bs{A}_y(t)}$.
  The time spent in the current level is even smaller, since the last visit has not ended.
  Thus we have the bound
  \[ t \leq \sum_y \cL_{y,\bs{A}_y(t)}. \]

  Let
  \[ N_t := 2\sqrt{R_t\log\log (R_t \vee 4)} \geq \sqrt{R_t}. \]
  Obviously, $R_t\ge 4$ eventually.
  By the event of \cref{L:LIL}, for some $t_0$ and all $t\geq t_0$, only levels with $|y|\leq N_t$ have been visited, and each level has been visited at most $N_t$ times, so $|\bs{A}_y(t)|\le N_t$ for all $t\ge t_0$.
  For some $t_1$ the event $G_{N_t}$ does not occur for $t\geq t_1$.
  If the maximal absolute value of the $X$ coordinate of the BERW up to some time $t\ge t_1\vee t_0$ is $M$, then we have $M \leq R_t \leq N_t^2$.
  Thus every entrance point in every level $y$ is within $[-N_t^2,N_t^2]$.
  Therefore, since $D'_{N_t,y}$ does not occur, 
  \[
    \cL_{y,\bs{A}_y(t)} \leq 2^7|\bs{A}_y(t)|N_t^{3/2}\sqrt{\log N_t}
  \]
  for every $y$.

  Observe that
  \[ N_t^{3/2} \sqrt{\log N_t} \leq 4 R_t^{3/4} (\log t)^{1/2} (\log\log t)^{3/4}, \]
  which follows from the definition of $N_t$ and the fact that a.s. $R_t\le t$.
  We find that for $t\ge t_0\vee t_1$,
  \begin{align*}
    t &\leq \sum_{y} \cL_{y,\bs{A}_y(t)} \\
    &\leq 2^7N_t^{3/2}\sqrt{\log N_t}\sum_{y}  |\bs{A}_y(t)| \\
    &= 2^7 N_t^{3/2}\sqrt{\log N_t}R_t \\
    &\leq 2^9 R_t^{7/4} (\log t)^{1/2} (\log\log t)^{3/4}.
  \end{align*}
  It follows that for $t\ge t_0\vee t_1$, 
  \[
    R_t \geq \frac{t^{4/7}}{2^{36/7} (\log t)^{2/7} (\log\log t)^{3/7}}. 
    \qedhere \]
\end{proof}

\bigskip

{\bf Acknowledgments.}
Alejandro Ram\'irez thanks the NYU-ECNU Institute of Mathematical
Sciences at NYU Shanghai where part of this work was done.

\medskip

\medskip

{\bf Funding.}
Omer Angel was funded in part by NSERC.  
Mark Holmes was supported by Future Fellowship FT160100166 from the Australian Research Council.  
Alejandro Ram\'irez has been partially supported by Fondecyt 1180259 and by Iniciativa Científica Milenio.

\medskip

\bibliographystyle{plain}

\end{document}